\documentclass[reqno]{amsart}

\usepackage{amsmath,amssymb,amsthm}
\usepackage{mathrsfs}
\usepackage{shapepar}
\usepackage{graphicx}
\usepackage[colorlinks=true,final]{hyperref}

\newtheorem{thm}{Theorem}[section]
\newtheorem*{mthm}{Theorem}
\newtheorem{lem}[thm]{Lemma}
\newtheorem{cor}[thm]{Corollary}
\newtheorem{prop}[thm]{Proposition}
\newtheorem{claim}[thm]{Claim}

\theoremstyle{definition}
\newtheorem{defn}[thm]{Definition}
\newtheorem{rmk}[thm]{Remark}
\allowdisplaybreaks

\theoremstyle{remark}
\newtheorem{ex}[thm]{Example}

\DeclareMathOperator{\lt}{lt}
\DeclareMathOperator{\ind}{Ind}

\DeclareMathOperator{\supp}{supp}
\DeclareMathOperator{\spn}{span}

\newcommand{\Ind}[2]{\ind_{#1}^{#2}}
\newcommand{\R}{\mathbb{R}}
\newcommand{\C}{\mathbb{C}}
\newcommand{\Hi}{\mathfrak{H}}
\newcommand{\units}{G^{(0)}}

\newcommand{\haars}{\{\lambda^{u}\}_{u\in\units}}

\newcommand{\haarv}[2]{\lambda^{#2}_{#1}}
\newcommand{\A}{\mathscr{A}}
\newcommand{\B}{\mathscr{B}}
\newcommand{\inv}{^{-1}}
\newcommand{\linner}[3]{{{}_{{}_{#1}}}\!\left\langle #2, #3\right\rangle}
\newcommand{\rinner}[3]{\left\langle #1, #2\right\rangle_{{}_{#3}}}

\newcommand{\sect}[2]{\Gamma_0(#1, #2)}

\title{Proper Actions of Groupoids on $C^*$-algebras}
\author{Jonathan Henry Brown}
\address{Department of Mathematics, Dartmouth College, Hanover, NH 03755}
\curraddr{Department of Mathematics, Ben Gurion University of the Negev, P.O.B. 653, Be'er Sheva 84105, Israel}
\email{jonathan.henry.brown@gmail.com}
\begin{document}
\maketitle
\begin{abstract}In 1990, Rieffel defined a notion of proper action of a group $H$ on a $C^*$-algebra $A$.  He then defined a generalized fixed point algebra $A^{\alpha}$ for this action and showed that $A^{\alpha}$ is Morita equivalent to an ideal of the reduced crossed product.  We generalize Rieffel's notion to define proper groupoid dynamical systems and show that the generalized fixed point algebra for proper groupoid actions is Morita equivalent to a subalgebra of the reduced crossed product.  We give some nontrivial examples of proper groupoid dynamical systems and show that if $(\A, G, \alpha)$ is a groupoid dynamical system such that $G$ is principal and proper, then the action of $G$ on $\A$ is saturated, that is  the generalized fixed point algebra in Morita equivalent to the reduced crossed product.

\end{abstract}
\begin{section}{Introduction}
In an effort to study deformation quantization of Poisson manifolds, Rieffel introduced in \cite{Rieffel90} a notion of proper group actions on $C^*$-algebras.  These actions are meant to behave like proper actions of groups on spaces.  To that end, he also defined a generalized fixed point algebra for proper dynamical systems which has some of the same properties as the generalized fixed point algebra, $C_0(G\backslash X)$, for a proper action of a group $G$ on a space $X$.  The main theorem of \cite{Rieffel90} shows that the generalized fixed point algebra for a proper dynamical system is Morita equivalent to an ideal of the reduced crossed product.  This generalizes a theorem of Green's \cite{Grn77} which gives a Morita equivalence between $C_0(X)\rtimes G$ and $C_0(G\backslash X)$ whenever $G$ acts freely and properly on $X$.  Since Rieffel introduced proper actions they have been studied in \cite{aHRW00}, \cite{aHRW03}, \cite{AHRW05},  \cite{KQR08}, \cite{MR09}, \cite{AKRW09}. 

In \cite{Rieffel90}, Rieffel also identifies a class of proper actions with the property that the generalized fixed point algebra is Morita equivalent to the reduced crossed product, Rieffel  calls these actions \emph{saturated}.  Saturated actions not only more closely resemble the situation in Green's theorem, they have also proved to be the most useful in applications \cite{aHRW00}, \cite{aHRW03}, \cite{AHRW05}.

The study of generalized fixed point algebras for proper group dynamical systems has lead to a wide range of interesting results in operator theory.  For example they have been used to prove results in nonabelian duality theory \cite{KQR08}, graph algebras \cite{aHRW03}, \cite{MR09} and the Brauer semigroup \cite{aHRW00}.   When one is interested in extending these results to the groupoid setting, one is naturally lead to seek an appropriate notion of a generalized fixed point algebra for groupoid dynamical systems and hence a notion of proper groupoid dynamical systems.

In this paper we propose a definition of proper groupoid dynamical systems and define a generalized fixed point algebra for these systems.    Our main theorem is as follows:
\begin{mthm}If a groupoid dynamical system is proper, then the generalized fixed point algebra for the action is Morita equivalent to a subalgebra of the reduced crossed product.\end{mthm}  Note that, this theorem  generalizes both \cite[Theorem 1.5]{Rieffel90} and \cite[Corollary 15]{Grn77}.  We also present some examples of proper groupoid dynamical systems and give conditions that guarantee that these examples are saturated.  To prove saturation in our examples, we needed to use a new averaging argument to overcome the fact that translations of open sets in groupoids are not necessarily open.  We believe this argument can be applied to prove other density results.  Along the way we recover a result in \cite{MW90} showing that $C^*(G)$ has continuous trace when $G$ is principal and proper.  Although our results are about reduced crossed products, the work of \cite{AR00} shows that the groupoids in our examples are amenable, so in these examples our results apply to the full crossed product.

We should note that there is some debate in the literature about the correct definition of proper \emph{group} dynamical systems \cite{Ex99},\cite{Meyer01},\cite{Rieff04}.  We have chosen to generalize Rieffel's original definition  \cite{Rieffel90} since (while it is not intrinsic) it gives a Morita equivalence result with the generalized fixed point algebra and is thus the definition most widely used in applications.  

We begin with a section on preliminaries which include a brief introduction to groupoid dynamical systems, induced representations and the reduced groupoid crossed product.  In Section 3 we state the definitions of proper groupoid dynamical systems and the generalized fixed point algebra and prove our main theorem.  Section 4 is devoted to fleshing out two examples and Section 5 is devoted showing freeness guarantees that these examples are saturated.  

\begin{subsection}{Conventions}Throughout this paper we will use the follow conventions.  If $A$ is a $C^*$-algebra, then $M(A)$ will denote the multiplier algebra of $A$ and $Z(A)$ will denote its center.  If $\pi:A\rightarrow B$ is nondegenerate, then $\bar{\pi}$ will denote its extension to $M(A)$.   If $X_1$ and $X_2$ are spaces with maps $\tau_i:X_i\rightarrow T$, then $X_1*X_2$ denotes the set $\{(x,y)\in X_1\times X_2:\tau_1(x)=\tau_2(y)\}$. Throughout, $G$ will denote a second countable locally compact Hausdorff groupoid with Haar system $\haars$.  We will use the notational conventions for groupoids established in \cite{Mcoord} which are the same as those in \cite{Ren80} except that we use $s$ to denote the source map.  If $G$ acts on a topological space $X$, then $X$ is fibred over $\units$ by a map $r_X$, furthermore there exists a map $\Phi:G*X\rightarrow X\times X$ given by $(\gamma,x)\mapsto (\gamma\cdot x,x)$.  We say that the action of $G$ on $X$ is \emph{free} if this map is injective and we say the action is \emph{proper} if $\Phi$ is a proper map.  Note that if $G$ acts properly on a locally compact Hausdorff space $X$, then the orbit space $X/G$ is locally compact and Hausdorff. We say $G$ is \emph{principal} if the natural action of $G$ on its unit space given by $\gamma\cdot s(\gamma)=r(\gamma)$ is free, we say $G$ is \emph{proper} if this action is proper.  We will show in Proposition \ref{ex scalar} that proper actions of groupoids on spaces give rise to proper groupoid dynamical systems as defined in Definition \ref{proper dyn sys}, so there should be no cause for confusion between the two uses of the word \emph{proper}.    Unless otherwise stated we will assume that all of our $C^*$-algebras are separable and all spaces $X$ are locally compact and Hausdorff.\end{subsection}
\end{section}
\begin{section}{Preliminaries}
\begin{subsection}{$C_0(X)$-algebras}
Groupoids must act on fibred objects, so to construct an appropriate notion of a groupoid dynamical system we need to have fibred $C^*$-algebras.  To that end, given a locally compact Hausdorff space $X$, a \emph{$C_0(X)$-algebra} is a $C^*$-algebra $A$ together with a nondegenerate homomorphism of $C_0(X)$ into $Z(M(A))$. $C_0(X)$-algebras are well studied objects in their own right, but for our needs it is enough to know that they have an associated fibred structure.  Specifically, if $C_{0,x}(X)$ is the set of functions in $C_0(X)$ vanishing at $x\in X$, then $I_x:=\overline{C_{0,x}(X)\cdot A}$ is an ideal in $A$ and $A(x):=A/I_x$ is called the fibre of $A$ over $x$.  The image of $a$ in $A(x)$ is denoted by $a(x)$, and the set $\{A(x):x\in X\}$ gives rise to an \emph{upper semicontinuous $C^*$-bundle} $\A$ over $X$ \cite[Theorem C.26]{TFB2}. 

\begin{defn}Let $X$ be a locally compact Hausdorff space, an {\emph{upper semicontinuous $C^*$-bundle}} over $X$ is a topological space $\A$ together with a continuous open surjection $p=p_{\A}:\A\rightarrow X$ such that  each fibre $A(x):=p\inv(\{x\})$ \index{$A(x)$} is a $C^*$-algebra and $\A$ satisfies the following axioms:
\begin{enumerate}
\item[B1] The map $a\mapsto \|a\|$ is upper semicontinuous from $\A$ to $\R^+$  (That is, for all $\epsilon>0$, $\{a\in \A: \|a\|<\epsilon\}$ is open),
\item[B2] The maps $(a,b)\mapsto a+b$ and $(a,b)\mapsto ab$ are continuous from $\A*\A$ to $\A$,
\item[B3] For each $k\in\C$, the maps $a\mapsto ka$ and $a\mapsto a^*$ are continuous from $\A$ to $\A$, and 
\item[B4] If $\{a_i\}$ is a net in $\A$ such that $p(a_i)\rightarrow x$ and $\|a_i\|\rightarrow 0$, then $a_i\rightarrow 0_x$ (where $0_x$ is the zero element of $A(x)$).
\end{enumerate}
\label{defn usccb}\end{defn}

The point is if we let $A=\sect{X}{\A}$ be the $C^*$-algebra of continuous sections of $\A$ vanishing at infinity then $A$ is a $C_0(X)$-algebra.  Throughout this paper we will denote bundles by script letters $\A$ and the corresponding section algebras by the corresponding Roman letter $A$.  For a more detailed discussion of $C_0(X)$-algebras the reader is encouraged to see \cite[Appendix C]{TFB2}.  
\end{subsection}
\begin{subsection}{The Reduced Crossed Product}\label{sec red cross}
\begin{defn}\label{def dyn sys}Suppose that $G$ is a second countable locally compact groupoid with Haar system $\haars$ and $\A$ is an upper semicontinuous $C^*$-bundle over $\units$. Suppose  the associated $C_0(X)$-algebra, $A=\sect{\units}{\A}$ is separable. An action $\alpha$ of $G$ on $A$ is a family of $*$-isomorphisms $\{\alpha_{\gamma}\}_{\gamma\in G}$ such that 
\begin{enumerate}
\item\label{dyn sys rs} for each $\gamma\in G$, $\alpha_{\gamma}:A(s(\gamma))\rightarrow A(r(\gamma))$,
\item\label{dyn sys alg}for all $(\gamma,\eta)\in G^{(2)}$, $\alpha_{\gamma\eta}=\alpha_{\gamma}\circ\alpha_{\eta}$,
\item\label{dyn sys cts}the map $(\gamma, a)\mapsto \alpha_{\gamma}(a)$ is a continuous map from $G*\A$ to $\A$.
\end{enumerate}
The triple $(\A,G,\alpha)$ is called a (groupoid) dynamical system.\end{defn}

The point is that given a dynamical system, we can construct a convolution algebra which we then complete to obtain the reduced crossed product.  The remainder of this section is devoted to a sketch of this construction.  First we need the following definition.
\begin{defn}Let $(\A, G,\alpha)$ be a groupoid dynamical system, we define the pullback bundle of $\A$ to be
\begin{equation}\label{r*A}r^*\A:=\{(\gamma,a):r(\gamma)=p_{\A}(a)\}.\end{equation}\end{defn}

\begin{prop}[{\cite[Proposition 4.4]{MW08}}]Let $G$ be a groupoid with Haar system $\haars$, if we define $\Gamma_c(G,r^*\A)$ to be the set of continuous compactly supported sections of $r^*\A$, then  $\Gamma_c(G,r^*\A)$ is a $*$-algebra with respect to the operations
\begin{equation*}f*g(\gamma):=\int_Gf(\eta)\alpha_{\eta}\left(g(\eta\inv\gamma)\right)d\haarv{}{r(\gamma)}(\eta) \text{~and~}f^*(\gamma)=\alpha_{\gamma}\left(f(\gamma\inv)^*\right).\end{equation*}\label{eq conv struc}\end{prop}

The goal is to complete this convolution algebra in the norm induced by regular representations. Since we use regular representations extensively in the sequel we will sketch their construction here.  To continue we need the notion of a Borel Hilbert bundle.  For our purposes a Borel Hilbert bundle $X*\Hi$ over $X$ is bundle of Hilbert spaces $X*\Hi=\{\mathcal{H}(x)\}_{x\in X}$ along with a Borel structure satisfying some technical conditions (see \cite[Definition F.1]{TFB2}). Given a measure $\mu$ on $X$ we can form the Hilbert space $L^2(X*\Hi, \mu)$ in the obvious way.  This Hilbert space is just the direct integral $\int_X^{\oplus}\mathcal{H}(x)d\mu(x)$ and gives us the notion of a fibred Hilbert space that we need for groupoid representations. 

Suppose $\pi$ is a (separable) $C_0(\units)$-linear representation of $A$ on $\mathcal{H}_{\pi}$.  Then by \cite[Proposition F.26]{TFB2} there exists a Borel Hilbert bundle $\units*\Hi$, a measure $\mu_{\pi}=\mu$ on $\units$ (note: $\mu$ need not be quasi invariant) and a Borel family of representations $\{\pi_u\}_{u\in\units}$ of $A$ on $\mathcal{H}(u)$ such that $\pi$ is unitarily equivalent to the representation 
\begin{equation}\rho=\int_{\units}^{\oplus}\pi_u d\mu(u) \quad\text{given by}\quad\rho(a)h(u)=\pi_u(a)h(u).  \label{eq dir int}\end{equation}
It is not hard to see from the proof of \cite[Proposition F.26]{TFB2} that $I_u\subset\ker(\pi_u)$ $\mu$-almost everywhere so that $\pi_u$ descends to a well defined representation on $A(u)$.
Therefore
\begin{equation}\pi_u(a)h(u)=\pi_u(a(u))h(u) \quad\mu\text{-almost everywhere}.\label{eq piu well def}\end{equation}
We can then form the pull-back Hilbert bundle $s^*(\units*\Hi)=:G*_s\Hi$ and  define the measure $\nu\inv=\int_{\units}\haarv{u}{}d\mu$ to form a new Hilbert space $L^2(G*_s\Hi,\nu\inv)$.  Note that functions $h\in L^2(G*_s\Hi,\nu\inv)$ have the property that $h(\gamma)\in \mathcal{H}(s(\gamma))$.  So that, 
\begin{equation}\label{eq def indpi}\Ind{}{}\pi(f)h(\gamma)=\int_G \pi_{s(\gamma)}\left(\alpha_{\gamma}\inv(f(\eta))\right)h(\eta\inv\gamma)\haarv{}{r(\gamma)}(\eta)\end{equation}
defines a representation of  $\Gamma_c(G, r^*\A)$ induced by $\pi$  on $L^2(G*_s\Hi,\nu\inv)$.
We call these representations \emph{regular} and define the reduced norm on $\Gamma_c(G,r^*\A)$ to be 
\begin{equation}\|f\|_r:=\sup\{\|\ind\pi(f)\|:
\pi \text{~is a $C_0(\units)$-linear representation of $A$}\}.\label{eq red norm}\end{equation}

\begin{rmk}This definition agrees with those given in \cite{Mcoord} and \cite{Ren80}, but is \emph{a priori} different from that given in \cite{AR00}.  We suspect that all of these definitions agree, but have yet to prove it.  But the set of regular representations used in \cite{AR00} is the subset of the regular representations defined above such that $\mu_{\pi}$ is a point mass measure. Thus $\|\cdot\|_r$ is greater than or equal to the norm $\|\cdot\|_{\text{red}}$ considered in \cite{AR00} which is enough for our purposes.\end{rmk}

As usual we can  define the reduced crossed product of a dynamical system $(\A,G, \alpha)$, denoted $\A\rtimes_{\alpha,r}G$, to be the completion of $\Gamma_c(G,r^*\A)$ in the norm $\|\cdot\|_r$.

In this paper we will also use the $I$-norm on $\Gamma_c(G,r^*\A)$ given by 
$$\|f\|_I:=\max\left\{\sup\left\{\int\|f\|d\lambda_u\right\},\sup\left\{\int\|f\|d\lambda^u\right\}\right\}.$$
We denote the completion of $\Gamma_c(G,r^*\A)$ in this norm by $L^I(G,r^*\A)$.
\end{subsection}
\end{section}

\begin{section}{Proper Actions}

\begin{subsection}{Defining Proper Dynamical Systems}
The following definition is modeled after \cite[Definition 1.2]{Rieffel90}.
\begin{defn}\label{proper dyn sys}\index{proper dynamical system} Suppose $(\A, G, \alpha)$ is a groupoid dynamical system and let $A=\sect{\units}{\A}$ be the associated $C_0(\units)$-algebra.  We say that the dynamical system $(\A, G, \alpha)$ is {\emph{proper}} if there exists a dense $*$-subalgebra $A_0\subset A$, such that the following two conditions hold:
\begin{enumerate}
\item \label{def prop 1}For all $a,b\in A_0$, the function $\linner{E}{a}{b}:\gamma\mapsto a(r(\gamma))\alpha_{\gamma}\left(b(s(\gamma))^*\right)$ is integrable. That is, the function $\gamma\mapsto \|\linner{E}{a}{b}(\gamma)\|$ is in $L^I(G)$.\index{$M(A_0)^{\alpha}$}
\item \label{def prop 2}Let \begin{equation} \label{eqM(A0)alpha} M(A_0)^{\alpha}:=\left\{d\in M(A): A_0d \subset A_0, ~\bar{\alpha}_{\gamma}(d(s(\gamma)))=d(r(\gamma)) \right\}.\end{equation}
Then for all $a,b\in A_0$, there exists a unique element $\rinner{a}{b}{D}\in M(A_0)^{\alpha}$ such that for all $c\in A_0$
 $$(c\cdot \rinner{a}{b}{D})(u)=\int_G c(r(\gamma))\alpha_{\gamma}\left(a^*b(s(\gamma))\right)d\haarv{}{u}(\gamma).$$
 \end{enumerate}
 \end{defn}

For a proper dynamical system, $(\A, G, \alpha)$, we denote $\spn\{\linner{E}{a}{b}:a,b\in A_0\}$ by $E_0$.  Now since the functions $\linner{E}{a}{b}$ are integrable, $E_0\subset \A\rtimes_{\alpha,r}G$ and we denote $E=\overline{E_0}$ in $A\rtimes_{\alpha,r}G$.

\begin{rmk}One may wonder at first why we chose $A_0\subset A$ instead of $A_0\subset \A$.  But since condition \eqref{def prop 1} is a condition about integrability of sections, $A_0$ had to be a subset of the section algebra instead of a subset of the bundle.\end{rmk}

\begin{rmk}In \cite[Section 1]{Rieffel90}, for a group dynamical system $(B,H, \beta)$, Rieffel defines 
$M(B_0)^{\beta}:=\left\{d\in M(B): B_0d \subset B_0, ~\bar{\beta}_{\gamma}(d)=d\right\}.$  That is $M(B_0)^{\beta}$ is the set of $\beta$-invariant elements of $M(B)$ that map $B_0$ to itself.  However, in a groupoid dynamical system $(\A,G, \alpha)$, $\alpha_{\gamma}:A(s(\gamma))\rightarrow A(r(\gamma))$, thus if $s(\gamma)\neq r(\gamma)$ then $\bar{\alpha}_{\gamma}(c)$ can not equal $c$ for $c\in M(A(s(\gamma)))$.  However, if $d\in M(A)$, then $d$ fibres over $\units$ and $\alpha_{\gamma}\left(d(s(\gamma))\right)$ acts on $A(r(\gamma))$.  So we will call $d\in M(A)$ $\alpha$-invariant if $\bar{\alpha}_{\gamma}(d(s(\gamma)))=d(r(\gamma))$ for all $\gamma\in G$. This is how we define it in \eqref{eqM(A0)alpha}.

To see that this is a reasonable definition, first note that if $G$ is a group then $r(\gamma)=s(\gamma)=e$ for all $\gamma\in G$.  Thus $\bar{\alpha}_{\gamma}(d(s(\gamma)))=d(r(\gamma))$ reduces to $\bar{\alpha}_{\gamma}(d)=d$, which is the definition of $\alpha$-invariant in the group case.  Also compare to \cite[Proposition 3.11]{AR00} and consider the following example.
\begin{ex} Let $A=C_0(\units)$, then $A$  is a $C_0(\units)$-algebra and the associated upper semicontinuous $C^*$-bundle is $\mathscr{T}:=\units\times \C.$  Let $G$ act on $\mathscr{T}$ by left translation, that is $\lt_{\gamma}(s(\gamma),\xi)=(r(\gamma),\xi)$.    Now $M(A)=\Gamma^b(\units,\mathscr{T})$, so if $d\in M(A)$ is $\lt$-invariant, then $(r(\gamma),d(r(\gamma)))=\bar{\lt}_{\gamma}((s(\gamma), d(s(\gamma))))=(r(\gamma),d(s(\gamma)))$. That is $d$ is constant on orbits and we can view $d$ as a function in $C^b(G\backslash\units)$.
\label{ex d constant orbits}\end{ex}

\end{rmk}

\begin{ex}\label{ex proper group dyn sys}Suppose $(B,H,\beta)$ is a proper group dynamical system with respect to the subalgebra $B_0$ as in \cite[Definition 1.2]{Rieffel90}.  Then $(B,H,\beta)$ is a proper groupoid dynamical system with respect to Definition \ref{proper dyn sys} once we make the standard allowances for the lack of modular function in the groupoid definition. 
\end{ex}
\begin{rmk}\cite[Definition 1.2]{Rieffel90} has an extra condition that we do not assume in Definition \ref{proper dyn sys}.  He assumes that $\beta_s(A_0)\subset A_0$ for all $s$ in the group $H$ where $(B, H,\beta)$ is a group dynamical system.  This assumption allows him to show that  $E$ is an ideal in the reduced crossed product.  Unfortunately, we have not yet been able to find a well defined analogous condition for groupoid dynamical systems.  This means a group dynamical system  can be a proper groupoid dynamical system without being a proper group dynamical system under \cite[Definition 1.2]{Rieffel90}.  
\end{rmk}

\begin{lem} If $(\A,G,\alpha)$ is a proper dynamical system then the action 
\begin{equation}(f\cdot c)(u):=\int_G f(\gamma)\alpha_{\gamma}(c(s(\gamma)))d\haarv{}{u}(\gamma).\label{eq li act}\end{equation}
for $f\in L^I(G, r^*\A)$ and $c\in A$, and inner product in condition \eqref{def prop 1} of Definition \ref{proper dyn sys}  define a  pre-Hilbert module structure on ${_{E_0}}\!A_0$.
\label{lem hil mod}\end{lem}

\begin{proof}First note that $\|\gamma\mapsto f(\gamma)\alpha_{\gamma}(c(s(\gamma)))\|\leq \|f\|_I\|c\|$ so that the action is bounded.  The linear and adjoint relations are routine. To show $E_0$ is a subalgebra of $A\rtimes_{\alpha,r} G$ and the action of $E_0$ to commutes with the inner product,  we perform the following computation. For $a,b,c,d\in A_0$,
\begin{align*} &\linner{E}{a}{b}*\linner{E}{c}{d}(\gamma)=\int_G \linner{E}{a}{b}(\eta)\alpha_{\eta}\left(\linner{E}{c}{d}(\eta\inv\gamma)\right)d\haarv{}{r(\gamma)}(\eta)\\
&=\int_G \linner{E}{a}{b}(\eta)\alpha_{\eta}\left(c(s(\eta))\alpha_{\eta\inv\gamma}(d(s(\gamma))^*)\right)d\haarv{}{r(\gamma)}(\eta)\\
&=\Bigl(\int_G \linner{E}{a}{b}(\eta)\alpha_{\eta}(c(s(\eta)))d\haarv{}{r(\gamma)}(\eta)\Bigr)\alpha_{\gamma}(d(s(\gamma))^*)=\linner{E}{\linner{E}{a}{b}\cdot c}{d}(\gamma).
\end{align*}

It remains to show that the inner product is positive, for this we use the following lemma.
\begin{lem}\label{lem ind inner}Suppose $(\A, G, \alpha)$ is a groupoid dynamical system and $\pi$ is a $C_0(\units)$-linear representation of $A=\sect{\units}{\A}$ so that we can decompose $\pi$ as in \eqref{eq dir int}. If $a\in A_0$ and $h\in L^2(G*_s\Hi,\nu\inv)$ then
\begin{align}\langle&\ind \pi(\linner{E}{a}{a})h,h\rangle\nonumber\\
&=\int_{\units}\Bigl\langle\int_G\pi_{u}\bigl(\alpha_{\eta\inv}\inv(a(s(\eta))^*)\bigr)h(\eta\inv)d\haarv{}{u}(\eta),\int_G\pi_{u}\bigl(\alpha_{\gamma\inv}\inv\left(a(s(\gamma)^*)\right)\bigr)\Bigr.\nonumber\\
&\hspace{2.5in}\Bigl.h(\gamma\inv)d\haarv{}{u}(\gamma)\Bigr\rangle_{\mathcal{H}(u)}(\gamma)d\mu(u)\label{eq ind inner}\end{align}
\end{lem}
\begin{proof} We now compute:
\begin{align*}&\left\langle\Ind{}{}\pi\left(\linner{E}{a}{a}\right)h,h\right\rangle =\int_G \rinner{\Ind{}{}\pi\left(\linner{E}{a}{a}\right)h(\gamma)}{h(\gamma)}{\mathcal{H}(s(\gamma))}d\nu\inv(\gamma)\\
&=\int_G\int_G\left\langle\pi_{s(\gamma)}\left(\alpha_{\gamma}\inv\bigl(a(r(\eta))\alpha_{\eta}(a(s(\eta))^*)\bigr)\right)h(\eta\inv\gamma)\right.,\\
&\hspace{2.5 in}\bigl.h(\gamma)\bigr\rangle_{\mathcal{H}(s(\gamma))}d\haarv{}{r(\gamma)}(\eta)d\nu\inv(\gamma)\\
&=\int_G\int_G\left\langle\pi_{s(\gamma)}\bigl(\alpha_{\eta\inv\gamma}\inv(a(s(\eta))^*)\bigr)h(\eta\inv\gamma),\pi_{s(\gamma)}\left(\alpha_{\gamma}\inv\left(a(r(\gamma)^*)\right)\right)\right.\\
&\hspace{2.5in}\Bigl.h(\gamma)\Bigr\rangle_{\mathcal{H}(s(\gamma))}d\haarv{}{r(\gamma)}(\eta)d\nu\inv(\gamma).
\end{align*}

Now using the left invariance of the Haar system to replace $\eta$ with $\gamma\eta$ the above becomes 
\begin{align*}&=\int_G\Bigl\langle\int_G\pi_{s(\gamma)}\bigl(\alpha_{\eta\inv}\inv(a(s(\eta))^*)\bigr)h(\eta\inv)d\haarv{}{s(\gamma)}(\eta),\pi_{s(\gamma)}\left(\alpha_{\gamma}\inv\left(a(r(\gamma)^*)\right)\right)\Bigr.\\
&\hspace{2.5in}\Bigl.h(\gamma)\Bigr\rangle_{\mathcal{H}(s(\gamma))}d\nu\inv(\gamma).\end{align*}
But now $s(\gamma)=r(\eta)$ so the above becomes
\begin{align*}&=\int_G\Bigl\langle\int_G\pi_{r(\eta)}\bigl(\alpha_{\eta\inv}\inv(a(s(\eta))^*)\bigr)h(\eta\inv)d\haarv{}{s(\gamma)}(\eta),\pi_{s(\gamma)}\left(\alpha_{\gamma}\inv\left(a(r(\gamma)^*)\right)\right)\Bigr.\\
&\hspace{2.5in}\Bigl.h(\gamma)\Bigr\rangle_{\mathcal{H}(s(\gamma))}d\nu\inv(\gamma)\\
&=\int_G\Bigl\langle\int_G\pi_{r(\eta)}\bigl(\alpha_{\eta\inv}\inv(a(s(\eta))^*)\bigr)h(\eta\inv)d\haarv{}{r(\gamma)}(\eta),\pi_{r(\gamma)}\bigl(\alpha_{\gamma\inv}\inv\left(a(s(\gamma)^*)\right)\bigr)\Bigr.\\
&\hspace{2.5in}\Bigl.h(\gamma\inv)\Bigr\rangle_{\mathcal{H}(r(\gamma))}d\nu(\gamma).\end{align*}
By decomposing $\nu$ and noticing $r(\eta)=r(\gamma)=u$  the above is equal to 

\begin{align*}&=\int_{\units}\Bigl\langle\int_G\pi_{u}\bigl(\alpha_{\eta\inv}\inv(a(s(\eta))^*)\bigr)h(\eta\inv)d\haarv{}{u}(\eta),\int_G\pi_{u}\bigl(\alpha_{\gamma\inv}\inv\left(a(s(\gamma)^*)\right)\bigr)\Bigr.\\
&\hspace{2.5in}\Bigl.h(\gamma\inv)d\haarv{}{u}(\gamma)\Bigr\rangle_{\mathcal{H}(u)}(\gamma)d\mu(u).\qedhere\end{align*}
\end{proof}
Now Lemma \ref{lem ind inner} gives 
$\ind\pi(\linner{E}{a}{a})$ is positive since $\mu$ is a positive measure. This holds for all induced representations,  so $\linner{E}{a}{a}$ is positive as an element of $\A\rtimes_{\alpha,r}G$ and hence as an element of $E$, so that ${_{E_0}}\!A_0$ is a pre-Hilbert module and thus completes to a Hilbert $E$-module. \end{proof}

\end{subsection}
\begin{subsection}{Morita Equivalence}
\begin{thm}Let $(\A, G, \alpha)$ be a proper dynamical system with respect to $A_0$, and let $D_0=\spn\{\rinner{a}{b}{D}:a,b\in A_0\}$\index{$D_0$}.  Viewing $E_0$ as a dense subalgebra of  $E:=\overline{E_0}^{\A\rtimes_{\alpha,r}G}$, and $D_0$ as a dense subalgebra of  $A^{\alpha}:=\overline{D_0}^{M(A)}$, then $A_0$ equipped with the $E_0$-action defined in equation \eqref{eq li act} and inner products defined in Definition \ref{proper dyn sys} is a $E_0-D_0$ pre-imprimitivity bimodule.\label{thm proper morita}\end{thm}  

\begin{rmk} We call $A^{\alpha}$\index{$A^{\alpha}$} the generalized fixed point algebra \index{generalized fixed point algebra} for the dynamical system $(\A, G, \alpha)$.  So that Theorem \ref{thm proper morita} gives that the generalized fixed point algebra  is Morita equivalent to a subalgebra of the reduced crossed product.\end{rmk}

\begin{proof}[Proof of Theorem \ref{thm proper morita}] The proof of this theorem follows \cite[Section 1]{Rieffel90} fairly closely.  From Lemma \ref{lem hil mod}, $A_0$ is a pre-Hilbert $E_0$-module. The goal is to show that $A^{\alpha}$ is the imprimitivity algebra  for the resulting Hilbert module. From the definition of the $D_0$-valued inner product and the definition of the $E_0$-action it is easy to see that  $D_0$ satisfies the algebraic conditions.  

It remains to show that the $D_0$-action is bounded and adjointable so that $D_0\subset \mathcal{L}(_{E}\!\overline{A_0})$  and furthermore, that in fact the norm of $d\in D_0$ as an element of $\mathcal{L}(_{E}\!\overline{A_0})$ coincides with its norm as an element of $M(A)$.  The last statement ensures that $A_0$ completes to an $E-A^{\alpha}$-imprimitivity bimodule.  

First, we show that the action of $M(A_0)^{\alpha}$ on $A_0$ is bounded.  Let $\pi$ be a $C_0(\units)$-linear representation of $A$, and $\Ind{}{}\pi$ be the  corresponding  representation of the reduced crossed product.  Using $ad$ in Lemma \ref{lem ind inner} we get 
\begin{align}\langle\ind \pi&(\linner{E}{ad}{ad}h),h\rangle\nonumber\\
&=\int_{\units}\Bigl\langle\int_G\pi_{u}\left(\alpha_{\eta\inv}\inv((ad(s(\eta)))^*)\right)h(\eta\inv)d\haarv{}{u}(\eta),\Bigr.\nonumber\\
&\hspace{.5in}\Bigl.\int_G\pi_{u}\left(\alpha_{\gamma\inv}\inv\left((ad(s(\gamma))^*)\right)\right)h(\gamma\inv)d\haarv{}{u}(\gamma)\Bigr\rangle_{\mathscr{H}(u)}d\mu(u)\label{eq ind ad inner}\end{align}
Using the fact that $r(\gamma)=r(\eta)=u$ and $d(s(\gamma))=d(r(\gamma))$ \eqref{eq ind ad inner} is equal to 
\begin{align*}&=\int_{\units}\Bigl\langle\int_G\bar{\pi}_{u}(d(u)^*)\pi_{r(\eta)}\bigl(\alpha_{\eta\inv}\inv \bigl(a(s(\eta))^*\bigr)\bigr)h(\eta\inv)d\haarv{}{u}(\eta),\Bigr.\\
&\hspace{.7 in}\Bigl. \int_G\bar{\pi}_{u}(d(u)^*)\pi_{r(\gamma)}\bigl(\alpha_{\gamma\inv}\inv\bigl(a(s(\gamma))^*\bigr)\bigr)h(\gamma\inv)d\haarv{}{u}(\gamma)\Bigr\rangle_{{}_{\mathcal{H}(u)}}d\mu(u)\\
&\leq\int_{\units}\Bigl\langle\|d\|_{M(A)}^2\int_G\pi_{r(\eta)}\bigl(\alpha_{\eta\inv}\inv (a(s(\eta))^*)\bigr)h(\eta\inv)d\haarv{}{u}(\eta),\Bigr.\\
&\hspace{.7 in}\Bigl. \int_G\pi_{r(\gamma)}(\alpha_{\gamma\inv}\inv(a(s(\gamma))^*))h(\gamma\inv)d\haarv{}{u}(\gamma)\Bigr\rangle_{{}_{\mathcal{H}(u)}}d\mu(u)\\
&=\|d\|_{M(A)}^2\linner{}{\linner{E}{a}{a}h}{h}.
\end{align*}
Here the last inequality follows from the fact that $\mu$ is a positive measure and the integrand is positive. Since this holds for all induced representations of $\A\rtimes_{\alpha,r}G$ we have $\|ad\|_{A_0}\leq\|a\|_{A_0}\|d\|_{M(A)}$. 

It is not hard to see that $d^*$ is the adjoint for $d$ as an element of $L({_{_{E}}}\!\overline{A_0})$, so that $d$ extends to an adjointable operator on $_E\!\overline{A_0}$.

It remains to show that the norm of $d$ as an element of $M(A)$ is the same as the norm of $d$ as an element of $\mathcal{L}(_E\!\overline{A_0})$.  We do this by finding a representation $\pi$ of $A$ and  constructing an  $a\in A_0$ and $h\in L^2(G*_s\mathfrak{H})$ such that $\rinner{\ind \pi(\linner{E}{ad}{ad})h}{h}{}$ is close to $\|d\|^2$ and $\rinner{\ind\pi(\linner{E}{a}{a})h}{h}{}$ is close to $1$.  It follows that 
$\rinner{\ind\pi(\linner{E}{ad}{ad}h)}{h}{}$  is close to $\|d\|^2\rinner{\ind\pi(\linner{E}{a}{a})h}{h}{}$.  

If $\pi:A\rightarrow L^2(\units*\Hi)$ is a faithful representation of $A$ the idea due to Rieffel \cite[p.~151]{Rieffel90} is the following.  We first pick $v\in L^2(\units*\Hi)$ such that $\bar{\pi}(d)v$ is close to $\|d\|$, and find an $a\in A_0$ such that $\pi(a)v$ is close to $v$.  We then let $h$ be a vector in $L^2(G*_s\Hi)$ that extends $v$ and is supported on a small neighborhood of $\units$.  The calculation used to show $d$ is bounded will then also show that $\rinner{\ind\pi(\linner{E}{ad}{ad})h}{h}{}$ is close to $\|d\|^2$ and $\rinner{\ind\pi(\linner{E}{a}{a})h}{h}{}$ is close to $1$.  We are left with checking the technical details.  

Let $d\in M(A_0)^{\alpha}$ be given and let $\epsilon>0$ be small.  Suppose that $\pi$ is a faithful nondegenerate representation of $A$.    Then $\bar{\pi}$ is a faithful nondegenerate representation of $M(A)$.  Thus there exists $v\in \mathcal{H}_{\pi}$ such that \begin{equation}\label{eq dv sq}\|\bar{\pi}(d)v\|^2+\epsilon/6>\|d\|_{M(A)}^2.\end{equation}  Now  there exists a Borel Hilbert bundle $\units*\mathfrak{H}$ and a measure $\mu$ on $\units$ such that $\pi\cong\int_{\units}^{\oplus}\pi_ud\mu(u)$.  We identify $\pi$ with its direct integral and $\mathcal{H}_{\pi}$ with $L^2(\units*\mathfrak{H})$ and a nasty computation now shows we can assume that $v$ under this identification has compact support $K_v$.

Now pick $a_0$ close to an approximate unit of $A$ such that 
\begin{align}\|\pi((a_0d)^*)v\|^2+\epsilon/6&>\|\bar{\pi}(d^*)v\|^2,\label{eq a0dv sq}\\
\|\pi(a_0)v\|^2+\epsilon/(6\|d\|^2)&>\|v\|^2=1, \text{~and}\label{eq av close to v}\\
\|a_0\|&<1.\label{eq a close to 1}\end{align}  

We will use a constant multiple of $a_0$ as our $a$.  Before we state the next lemma we need a definition. A subset $L$ of a topological groupoid $G$ is called \emph{$s$-relatively compact}, if $L\cap s\inv(K)$ is relatively compact for every compact subset $K\subset\units$.  $r$-relatively compact subsets are defined similarly. A  compactness argument shows the next lemma which we will use to find an appropriate small neighborhood of $\units$.

\begin{lem}\label{lem app act} Let $(\A,G, \alpha)$ be a groupoid dynamical system, and $a\in \sect{\units}{\A}$.  Fix $\epsilon>0$, then there exists an open neighborhood $V$ of $\units$ in $G$ such that $V$ is both $r$- and $s$-relatively compact and for all $\gamma\in G$, $\|\alpha_{\gamma}(a(s(\gamma)))-a(r(\gamma))\|<\epsilon$.\label{lem act close}\end{lem}

Using Lemma \ref{lem act close} pick a symmetric $r,s$-relatively compact open neighborhood $V_{\epsilon}$ of $\units$ such that for all $\gamma\in V_{\epsilon}$
\begin{equation}\|\alpha_{\gamma}(a_0^*(s(\gamma)))-a_0^*(r(\gamma))\|<\epsilon/(12\|d\|^2).\label{eq act est}\end{equation}
Since $V_{\epsilon}$ is $s$-relatively compact,  $s\inv(u)\cap V_{\epsilon}$ is relatively compact.  Hence, 
$$\haarv{u}{}(V_{\epsilon})=\haarv{u}{}(s\inv(u)\cap V_{\epsilon})\leq \haarv{u}{}\left(\overline{s\inv(u)\cap V_{\epsilon}}\right)<\infty.$$
Furthermore, since $V_{\epsilon}$ is open and $u\in V_{\epsilon}$, we have $\lambda_u(V_{\epsilon})\neq 0$. Thus
\begin{equation}\label{G vect} \tilde{h}:=\frac{\chi_{{}_{V_{\epsilon}}}(\gamma)v(s(\gamma))}{\haarv{s(\gamma)}{}(V_{\epsilon})}\end{equation}
is defined and less than infinity for all $\gamma\in G$.  
\begin{claim}\label{clm til h} $\tilde{h}\in L^2(G*_s\mathfrak{H}, \nu\inv)$.\end{claim}
\begin{proof}  Now
\begin{align*}\|\tilde{h}\|_2^2&=\int_{\units}\int_G\Bigl\langle\frac{\chi_{{}_{V_{\epsilon}}}(\gamma)v(s(\gamma))}{\haarv{s(\gamma)}{}(V_{\epsilon})},\frac{\chi_{{}_{V_{\epsilon}}}(\gamma)v(s(\gamma))}{\haarv{s(\gamma)}{}(V_{\epsilon})}\Bigr\rangle_{{}_{\mathcal{H}(s(\gamma))}}d\haarv{u}{}(\gamma)d\mu(u)\\
&=\int_{\units}\frac{\chi_{{}_{K_v}}}{(\haarv{u}{}(V_{\epsilon}))}\rinner{v(u)}{v(u)}{\mathcal{H}(u)}d\mu(u).
\end{align*}
So to show that $\tilde{h}\in L^2(G*_s\Hi, \nu\inv)$ it suffices to show that $\frac{\chi_{{}_{K_v}}}{\haarv{u}{}(V_{\epsilon})}\in L^{\infty}(\units,\mu)$.

Pick $\psi\in C_c(G)$ such that $\psi|_{K_v}\equiv 1$, $0\leq \psi\leq 1$, and $\supp(\psi)\subset V_{\epsilon}$.  Then by the properties of the Haar system the function
$$\lambda(\psi):u\mapsto\int_G \psi d\haarv{u}{}$$
is continuous.  So $\lambda(\psi)|_{K_v}$ has a minimum $m$.   Since $\psi(u)=1$, if $u\in K_v$ then $\int\psi d\haarv{u}{}>0$ so that $m>0$.  But $\psi\leq\chi_{{}_{V_{\epsilon}}}$, so for $u\in K_v$,  we have $m\leq \lambda(\psi)(u)\leq \haarv{u}{}(V_{\epsilon})$.  Thus $\frac{\chi_{K_v}}{(\haarv{u}{}(V_{\epsilon}))}\in L^{\infty}(\units,\mu)$, giving $\tilde{h}\in L^2(G*_s\mathfrak{H}, \nu\inv)$. \end{proof} Define:

\begin{equation*}h=\frac{\tilde{h}}{k}=\frac{\chi_{{}_{V_{\epsilon}}}(\gamma)v(s(\gamma))}{k\haarv{s(\gamma)}{}(V_{\epsilon})}\quad\text{and}\quad  a=ka_0.\end{equation*}  
The next claim uses this $h$ and $a$ to get the estimates we need to complete the proof.
\begin{claim}For $a,~ d, ~\pi,, ~h, ~v,$ and $\epsilon$ chosen as above, \begin{align}&\bigl|\rinner{\Ind{}{}\pi(\linner{E}{ad}{ad})h}{h}{} -\rinner{\pi((a_0d)^*)v}{\pi((a_0d)^*)v}{\mathcal{H}_{\pi}} \bigr|<\epsilon/6 \text{~and}\label{1}\\
&\bigl|\rinner{\Ind{}{}\pi(\linner{E}{a}{a})h}{h}{} -\rinner{\pi((a_0)^*)v}{\pi((a_0)^*)v}{\mathcal{H}_{\pi}} \bigr|<\epsilon/(6\|d\|^2).\label{2}\end{align}\end{claim}

\begin{proof} We will compute the estimate for \eqref{1}, the computation for \eqref{2} is exactly the same.  First note that \begin{align}\rinner{\pi((a_0d)^*)v}{\pi((a_0d)^*)v}{\mathcal{H}_{\pi}}
&=\int_{\units}\Bigl\langle\pi_u((a_0d)(u))\haarv{u}{}(V_{\epsilon})\pi_u((a_0d)^*(u))\nonumber\Bigr.\\
&\Bigl.\hspace{.5 in}\haarv{u}{}(V_{\epsilon})\frac{v(u)}{\haarv{u}{}(V_{\epsilon})},\frac{v(u)}{\haarv{u}{}(V_{\epsilon})}\Bigr\rangle_{\mathcal{H}(u)}d\mu(u).\label{eq piadv}
\end{align}
Using Lemma \ref{lem ind inner} with $ad$ and $h$ we compute:
\begin{align*}\langle\ind&\pi(\linner{E}{ad}{ad})h,h\rangle\\
&=\int_{\units}\Bigl\langle\int_G\pi_{u}\left(\alpha_{\eta\inv}((ad(r(\eta)))^*)\right)h(\eta)d\haarv{u}{}(\eta),\Bigr.\\
&\hspace{.6in}\Bigl.\int_G\pi_{u}(\alpha_{\gamma\inv}((ad(r(\gamma))^*)))h(\gamma)d\haarv{u}{}(\gamma)\Bigr\rangle_{\mathcal{H}(u)}d\mu(u)\\
&=\int_{\units}\Bigl\langle\bar{\pi}_u(d^*(u))\int_{V_{\epsilon}}\pi_{u}\left(\alpha_{\eta\inv}(a_0(r(\eta))^*)\right)d\haarv{u}{}(\eta)\frac{v(u)}{\haarv{u}{}(V_{\epsilon})},\Bigr.\\
&\hspace{.6in}\Bigl.\bar{\pi}_u(d^*(u))\int_{V_{\epsilon}}\pi_{u}(\alpha_{\gamma\inv}(a_0(r(\gamma)^*)))d\haarv{u}{}(\gamma)\frac{v(u)}{\haarv{u}{}(V_{\epsilon})}\Bigr\rangle_{\mathcal{H}(u)}d\mu(u).
\end{align*}
Now $\pi_{u}(\alpha_{\gamma\inv}(a_0(r(\gamma)^*)))\in B(\mathcal{H}(u))$ for all $\gamma$, the map $$\gamma\mapsto \pi_{u}(\alpha_{\gamma\inv}(a_0(r(\gamma)^*)))$$ is continuous, and $\overline{s\inv(u)\cap V_{\epsilon}}$ is compact, so  there exists an operator $L(u)\in B(\mathcal{H}(u))$ such that such that \begin{align*}L(u)&=\int_{V_{\epsilon}}\pi_{u}(\alpha_{\gamma\inv}(a_0(r(\gamma)^*)))d\haarv{u}{}(\gamma),{~~\text{giving}}\\
\rinner{\linner{E}{ad}{ad}h}{h}{}&=\int_{\units}\rinner{L(u)^*\bar{\pi}_u(dd^*(u))L(u)\frac{v(u)}{\haarv{u}{}(V_{\epsilon})}}{\frac{v(u)}{\haarv{u}{}(V_{\epsilon})}}{}d\mu(u).\end{align*}
Thus from \eqref{eq piadv},
\begin{align}|\langle {{}_{_{E}}}\!\langle &ad,ad\rangle h,h\rangle - \rinner{\pi((a_0d)^*)v}{\pi((a_0d)^*)v}{}|\nonumber\\
&\leq\int_{\units}\|L(u)^*\bar{\pi}_u(d(u))\bar{\pi}_u(d^*(u))L(u)-\pi_u(a_0dd^*a_0^*(u))(\haarv{u}{}(V_{\epsilon}))^2\|\cdot\nonumber\\
&\hspace{1in}\rinner{\frac{v(u)}{\haarv{u}{}(V_{\epsilon})}}{\frac{v(u)}{\haarv{u}{}(V_{\epsilon})}}{\mathcal{H}(u)}d\mu(u).\label{ineq L}
\end{align}
\begin{claim}\label{clm lu est}With $L, ~\pi_u, ~a_0, ~d, ~\epsilon,\text{and~}V_{\epsilon}$ as above $$\|L(u)^*\bar{\pi}_u(d(u))\bar{\pi}_u(d^*(u))L(u)-\pi_u(a_0dd^*a_0^*(u))(\haarv{u}{}(V_{\epsilon}))^2\|<\frac{\epsilon}{6}(\haarv{u}{}(V_{\epsilon}))^2.$$\end{claim}

\begin{proof}First note that $$\left\|L(u)\right\|\leq \int_{V_{\epsilon}}\|\pi_{u}(\alpha_{\gamma\inv}(a_0(r(\gamma)^*)))\|d\haarv{u}{}(\gamma)<\haarv{u}{}(V_{\epsilon})$$
since $\|a_0\|<1$ from equation \ref{eq a close to 1}.
An unenlightening computation now shows 
\begin{align}\|L(u)^*\bar{\pi}_u(d(u))\bar{\pi}_u(d^*(u))L(u)&-\pi_u(a_0dd^*a_0^*(u))(\haarv{u}{}(V_{\epsilon}))^2\|\nonumber\\
&\leq 2\|d\|^2\haarv{u}{}(V_{\epsilon})\|L(u)-\pi_u(a_0^*(u))\haarv{u}{}(V_{\epsilon})\|.\label{eq lu est}
\end{align}
\begin{align*}\text{Now,}\quad\|&L(u)-\pi_u(a_0^*(u))\haarv{u}{}(V_{\epsilon})\|\\
&=\|\int_{V_{\epsilon}}\pi_{u}(\alpha_{\gamma\inv}(a_0(r(\gamma)^*)))d\haarv{u}{}(\gamma)-\pi_u(a_0^*(u))\haarv{u}{}(V_{\epsilon})\|\\
&\leq \int_{V_{\epsilon}}\|\pi_{u}(\alpha_{\gamma\inv}(a_0(r(\gamma)^*)))-\pi_u(a_0^*(u))\|d\haarv{u}{}(\gamma)<\epsilon/(12\|d\|^2)\haarv{u}{}(V_{\epsilon})\end{align*}
by equation \eqref{eq act est}.  Thus using equation \eqref{eq lu est}, we have \begin{align*}\|L(u)^*\bar{\pi}_u(d(u))&\bar{\pi}_u(d(u)^*)L(u)\\
&-\pi_u(a_0dd^*a_0^*(u))(\haarv{u}{}(V_{\epsilon}))^2\|<\epsilon/6\cdot(\haarv{u}{}(V_{\epsilon}))^2.\qedhere\end{align*} \end{proof} 
Combining Claim \ref{clm lu est} with \eqref{ineq L} we get
\begin{align}|\langle\Ind{}{}\pi &(\linner{E}{ad}{ad})h,h\rangle - \rinner{\pi((a_0d)^*v}{\pi((a_0d)^*v}{}|\nonumber\\
&<\int_{\units}\epsilon/6\cdot(\haarv{u}{}(V_{\epsilon}))^2\rinner{\frac{v(u)}{\haarv{u}{}(V_{\epsilon})}}{\frac{v(u)}{\haarv{u}{}(V_{\epsilon})}}{\mathcal{H}(u)}d\mu(u)=\epsilon/6.\label{eq Ein adad a0dv}\end{align}
Giving equation \eqref{1}.\end{proof}
Thus by combining equations \eqref{eq dv sq},~\eqref{eq a0dv sq} and \eqref{1} we get
 \begin{equation}|\rinner{\Ind{}{}\pi(\linner{E}{ad}{ad})h}{h}{}-\|d\|_{M(A)}^2|<\epsilon/2.\label{3}\end{equation}
Similarly by combining equations \eqref{eq av close to v} and \eqref{2} we get
\begin{equation}|\rinner{\Ind{}{}\pi(\linner{E}{a}{a})h}{h}{}-1|<\epsilon/(2\|d\|_{M(A)}^2).\label{4}\end{equation}
Now equations \eqref{3} and \eqref{4} give
\begin{align*}
&|\rinner{\Ind{}{}\pi(\linner{E}{ad}{ad})h}{h}{}-\|d\|^2_{M(A)}\rinner{\Ind{}{}\pi(\linner{E}{a}{a})h}{h}{}|<\epsilon
\end{align*}
Thus  $\|d\|_{\mathcal{L}(_E\!\overline{A_0})}= \|d\|_{M(A)}$ as desired. This completes the proof of Theorem \ref{thm proper morita}.
\end{proof}
\end{subsection}
\end{section}

\begin{section}{Fundamental Examples}
\begin{prop}Suppose $G$ is a  groupoid acting properly on space $X$, then $(C_0(X),G, \lt)$ is a proper groupoid dynamical system  with respect to the dense subalgebra $C_c(X)$. Furthermore, $C_0(X)^{\lt}\cong C_0(G\backslash X).$\label{ex scalar}\end{prop}

Before proceeding we should note that if $\mathscr{C}$ is the the upper semicontinuous $C^*$-bundle associated to $C_0(X)$, it is not hard to see that the fibres of $\mathscr{C}$ are given by $\{C_0(r_X\inv(u))\}_{u\in\units}$. The action $\lt$ is then given by $\lt_{\gamma}(f)(x)=f(\gamma\inv\cdot x)$ for $x\in r_X\inv(r(\gamma))$ and $f\in C_0(r_X\inv(s(\gamma)))$. Furthermore, the bundle $\mathscr{C}$ is actually a \emph{continuous} $C^*$-bundle in that the map from $\mathscr{C}\rightarrow \C$ given by $c\mapsto \|c\|$ is continuous by \cite[Theorem 3.26]{TFB2}. 

Now to show Proposition \ref{ex scalar} we need to show:  
\begin{enumerate}
\item\label{cond 1 sc} For all $f,g\in C_c(X)$, the function \begin{align*}\linner{E}{f}{g}(\gamma):=f|_{r_X\inv(r(\gamma))} & \lt_{\gamma}\left(g^*|_{r_X\inv(s(\gamma))}\right)\\
&=\left(x\in r_X\inv(r(\gamma))\mapsto f(x)\overline{g(\gamma\inv\cdot x)}\right)\end{align*} is integrable.  
\item \label{cond 2 sc}If $f,g\in C_c(X)$,  there exists a function $\rinner{f}{g}{D}\in C^b(G\backslash X)\subset M(C_c(X))^{\lt}$, such that for all $h\in C_c(X)$
 \begin{equation}h\rinner{f}{g}{D}|_{r_X\inv(u)}=\int_G  h|_{r_X\inv(u)} \lt_{\gamma}\left(f^*g|_{r_X\inv(s(\gamma))}\right)d\haarv{}{u}(\gamma).\label{DinC0}\end{equation}
 \end{enumerate}

First we will show that $\linner{E}{f}{g}$ is integrable for $f,g\in C_c(X)$.  Consider the continuous function 
 \begin{align*}G*X&\rightarrow \C\\
  (\gamma,x)&\mapsto f(x)\overline{g(\gamma\inv\cdot x)}.
  \end{align*}
Using the properness of the $G$-action, it is not hard to see that this function has compact support and is hence integrable.  This gives that $\linner{E}{f}{g}$ is integrable. 

It remains to show property \eqref{cond 2 sc}.  It suffices to show for $F\in C_c(X)$, there exists a function $d\in C^b(G\backslash X)$ such that 
\begin{equation}h(x)d(x)=\int_G  h(x) F(\gamma\inv\cdot x)d\haarv{}{r_X(x)}(\gamma)\quad\forall~ h\in C_c(X),~~x\in X.\label{eq1}\end{equation}
Using the properness of the $G$-action,  a compactness argument shows the set $L:=\{\gamma\in G :F(\gamma\inv \cdot x)\neq 0\}$ is relatively compact for a  fixed $x\in X$,  and hence the function $\gamma\mapsto F(\gamma\inv\cdot x)$ is $\haarv{}{r_X(x)}$-integrable.   So we can define 
$$d(x):=\int_G  F(\gamma\inv\cdot x)d\haarv{}{r_X(x)}(\gamma).$$
This $d$ certainly satisfies equation \eqref{eq1}.  It remains to show that $d(x)\in C^b(G\backslash X)$.
For this, we will use the following stronger lemma from \cite{MRW87} which we will restate here for convenience.

\begin{lem}[{\cite[Lemma 2.9]{MRW87}}] Let $G$ act properly on the left of a locally compact Hausdorff space $X$, if $f\in  C_c(X)$, then
$$\lambda(f)([x])=\int_G f(\gamma\inv\cdot x)d \haarv{}{r_X(x)}(\gamma)$$ defines a map   of $C_c(X)$ onto $C_c(G\backslash X).$
\label{lem l(f) cts}
\end{lem}

Lemma \ref{lem l(f) cts} now guarantees that $d(x)=\lambda(F)(x)$ is in $C_c(G\backslash X)$. So condition \eqref{cond 2 sc} is satisfied and the action of $G$ on $C_0(X)$ by left translation is proper. Furthermore, the onto assertion of Lemma \ref{lem l(f) cts} gives that the generalized fixed point algebra, $C_0(X)^{\lt}$, is $C_0(G\backslash X)$.  

\begin{rmk} Suppose $X=\units$ in Proposition \ref{ex scalar},  since $r_{\units}$ is the identity map, the associated bundle $\mathscr{C}=\mathscr{T}=\units\times \C$.  Furthermore, by Theorem \ref{thm proper morita}, $C_0(G\backslash \units)$ is Morita equivalent to a subalgebra of $C_0(\units)\rtimes_{\lt,r}G\cong C^*_r(G)$.  In particular if $G=H\times X$ is the transformation group groupoid then $C_0(H\backslash X)$ is Morita equivalent to a subalgebra of $C_0(X)\rtimes_{\lt,r} H$.\end{rmk}

\begin{prop} Suppose $G$ is a proper groupoid. Let $(G,\A,\alpha)$ be a groupoid dynamical system, and $A$ be the $C_0(\units)$-algebra corresponding to $\A$, then $(G,\A,\alpha)$ is proper with respect to the subalgebra $A_0=C_c(\units)\cdot A$.\label{prpA}\end{prop}

\begin{rmk}Notice that this result is similar to \cite[Theorem 5.7]{Rieff04}.\end{rmk}

\begin{rmk} If $G=H\times X$ is a transformation group groupoid, then $G$ acts properly on its unit space if and only if $H$ acts properly on $X$.\end{rmk} 

\begin{proof}[Proof of Proposition \ref{prpA}] First note that $C_c(\units)\cdot A$ is dense in $A$.  To show that the dynamical system $(G,\A,\alpha)$ is proper, we first need to show that the functions 

\begin{align*}\linner{E}{f\cdot a}{ g\cdot b}:\gamma\mapsto f(r(\gamma))&a(r(\gamma))  \alpha_{\gamma}  \left( \overline{g(s(\gamma))}  b^*(s(\gamma))\right)\\ 
& \left(=f(r(\gamma))\overline{g(s(\gamma))}a(r(\gamma))\alpha_{\gamma}{\left( b^*(s(\gamma))\right)}\right)\end{align*}
are integrable for $a,b\in A$ and $f,g\in C_c(\units)$.  Using the properness of the $G$ it is not hard to see that these functions have compact support.  To finish showing that $\|\linner{E}{f\cdot a}{g\cdot b}\|_{_{I}}<\infty$, we use the following lemma.

\begin{lem}Let $G$ be a groupoid, $\B$ be an upper semicontinuous $C^*$-bundle over $G$ and suppose $f\in \Gamma_c(G,\B)$.  Then $\|f\|_I<\infty$.\label{lemcptLI}\end{lem}

To prove this we use the subsequent proposition, which follows from a standard compactness argument.

\begin{prop} Suppose $X$ is a locally compact Hausdorff space, and let $f:X\rightarrow \R_{\geq 0}$ be an upper semicontinuous function with compact support, then $\|f\|_{\infty}<\infty$.  \label{propusccptbd}\end{prop}

\begin{proof}[Proof of Lemma \ref{lemcptLI}]
Let $K$ be the support of $f$.  By Proposition \ref{propusccptbd} we know that $\|f\|_{\infty}<\infty$.  So 
\begin{equation*}\int_G \|f(\gamma)\|d\haarv{}{u}(\gamma)\leq \int_G \chi_{{}_K}(\gamma)\|f\|_{\infty}d\haarv{}{u}(\gamma)=\|f\|_{\infty}\haarv{}{u}(K).\end{equation*}
Now since $\sup\{\haarv{}{u}(K)\}< \infty$, we have  $\|f\|_I<\infty$ by symmetry.
\end{proof}

It remains to show that for $f\cdot a, g\cdot b\in C_c(\units)\cdot A$ that there exists an element $\rinner{f\cdot a}{g\cdot b}{D}\in M(C_c(\units)\cdot A)^{\alpha}$ such that for all $h\cdot c\in A_0$,
\begin{equation*}\bigl((h\cdot c)\rinner{f\cdot a}{g\cdot b}{D}\bigr)([u])
=\int_G (h\cdot c)(r(\gamma))\alpha_{\gamma}\left((f\cdot a)^*(s(\gamma))(g\cdot b)(s(\gamma))\right)d\haarv{}{u}(\gamma).\end{equation*}
For this we will follow \cite[Lemma 6.17]{tfb} and \cite[Lemma 3.5]{aHRW00}.  

\begin{rmk} \label{rmkactr*_X} Let  $(\A,G,\alpha)$ be a groupoid dynamical system, and suppose $G$ acts on the left of a locally compact Hausdorff space $X$.  We can define the pull back bundle 
$$r_X^*\A:=\{(x,a):r_X(x)=p_{\A}(a)\}$$
and a continuous action of $G$ on $r_X^*\A$ via
$$\alpha^{r_X}_{\gamma}(x,a)=(\gamma\cdot x, \alpha_{\gamma}(a)).$$\index{$\alpha^{r_X}$}
\end{rmk}

\begin{defn} Let $(\A,G,\alpha)$ be a groupoid dynamical system and suppose that $G$ acts on the left of a locally compact Hausdorff space $X$. \index{$\Ind{G}{X}(\A, \alpha)$} Define
\begin{align*}\Ind{G}{X}(\A, \alpha):=\{f\in \Gamma^b(X, &r_X^*\A):  f( x)  =(\alpha^{r_X}_{\gamma})\inv(f(\gamma\cdot x))\\
& \text{~and~} \bigl([u]\mapsto \|f(u)\|\bigr) \text{~vanishes at~}\infty\}.
\end{align*}
\label{indGunits}\end{defn}
To finish the proof of Proposition \ref{prpA}, we will show that $\Ind{G}{X}(\A, \alpha)\subset M(A)^{\alpha}$ and that for $a,b\in A_0$, there exists a $d\in \Ind{G}{X}(\A, \alpha)$ satisfying the required properties in Definition \ref{proper dyn sys}.
\begin{rmk} If $H$ is a group and $A$ is a $C^*$-algebra, $\Ind{H}{X}(A,\alpha)$ is normally defined as:   \begin{align*}\Ind{H}{X}(A, \alpha):=\{f\in C^b(X, A):  f( & x)  =\alpha_{s}\inv\bigl(f(s\cdot x)\bigr)\\ & \text{~and~} \bigl([u]\mapsto \|f(u)\|\bigr) \in C_0(H\backslash X)\}.\end{align*}   This definition doesn't make sense for  an upper semicontinuous $C^*$-bundle, since the norm is upper semicontinuous.  But in the group case, the continuity of $[u]\mapsto \|f(u)\|$ is implied by the  condition that $f\in C^b(X,A)$. So the important part of the condition  $\left([u]\mapsto \|f(u)\|\right)\in C_0(H\backslash X)$ is that the function $[u]\mapsto \|f(u)\|$ vanishes at infinity.
\end{rmk}

\begin{lem} Let $(\A,G,\alpha)$ be a groupoid dynamical system. For $f\in C_c(\units)$ and $a\in A=\sect{\units}{\A}$, then \begin{equation}\label{lambdaf}\lambda(f\cdot a)(u):=\int_G \alpha_{\gamma}\left(f(s(\gamma)a(s(\gamma)))\right)d\haarv{}{u}(\gamma)\end{equation}\index{$\lambda(f)$} gives a well-defined element of $\Ind{G}{\units}(\A,\alpha)$.\label{lem6.17}\end{lem}

The proof of this lemma is essentially the same as the proof in \cite[Lemma 6.17]{tfb} with some minor modifications.
\begin{claim}\label{clm ind is mult} If $f\in \Ind{G}{\units}(\A,\alpha)$ then $$m_f:a\mapsto \left(v\mapsto f(v)a(v)\right)$$ is a multiplier of $A=\sect{\units}{\A}$.  \end{claim}

\begin{proof} First note that $f(v)\in A(v)\subset M(A(v))$ and for $f\in \Gamma^b(\units, \A)$ and $a\in \sect{\units}{\A}$ we have that $\left(v\mapsto f(v)a(v)\right)\in  \sect{\units}{\A}$.  Similarly, $v\mapsto f(v)^*a(v)\in \Gamma_0(\units,\A)$. So \cite[Lemma C.11]{TFB2} implies that $f\in M(A)$.\end{proof}  

Claim \ref{clm ind is mult} and Lemma \ref{lem6.17} give that $\lambda(f\cdot a)\in M(A)$ for all $f\in C_c(\units)$ and $a\in A$.   Furthermore, since $A$ is a $C_0(\units)$-algebra, $C_c(\units)\subset Z(M(A))$, so if $m\in M(A)$, $g\in C_c(\units)$ and $b\in A$ then $m(g\cdot a)=g\cdot (ma)\in C_c(\units)\cdot A$.  Thus, $\lambda(f\cdot a)(C_c(\units)\cdot A)\subset C_c(\units)\cdot A$ and so $\lambda(f\cdot a)\in M(C_c(\units)\cdot A)$.  

Notice
\begin{align*}\alpha_{\eta}(\lambda(f\cdot a)(s(\eta))) &= \alpha_{\eta}\Bigl(\int_G \alpha_{\gamma}(f\cdot a(s(\gamma)))d\haarv{}{s(\eta)}(\gamma)\Bigr)\\
&=\int_G \alpha_{\eta\gamma}(f\cdot a(s(\eta\gamma)))d\haarv{}{s(\eta)}(\gamma)\\
&=\int_G \alpha_{\gamma}(f\cdot a(s(\gamma)))d\haarv{}{r(\eta)}(\gamma)=\lambda(f\cdot a)(r(\eta)).\end{align*}
Thus $\lambda(f\cdot a)\in M(A_0)^{\alpha}$.

Finally, we need to show that for $g\in C_c(\units)$ and $b\in A$, then $$\left((g\cdot b)\lambda(f\cdot a)\right)(u)=\int_G g(r(\gamma))b(r(\gamma))\alpha_{\gamma}(f\cdot a(s(\gamma)))d\haarv{}{u}(\gamma).$$
But this is just a straight forward calculation.

Notice that $(f\cdot a)(g\cdot b)=fg\cdot ab\in C_c(\units)\cdot A$.  Thus we can define $\rinner{f\cdot a}{g\cdot b}{D}:=\lambda((f\cdot a)^*(g\cdot b))$ and from the above argument this 
$\rinner{f\cdot a}{g\cdot b}{D}$ has the desired properties, making $(\A,G,\alpha)$ a proper dynamical system with respect to the subalgebra $C_c(\units)\cdot A$.  \end{proof}

\begin{rmk} The subalgebra $E\subset \A\rtimes_{\alpha,r} G$ guaranteed by Proposition \ref{prpA} and Theorem \ref{thm proper morita} is actually an ideal in $\A\rtimes_{\alpha,r} G$.  To see this, suppose $f\in \Gamma_c(G,r^*\A)$, and $a,b\in A_0$ then by a similar calculation to the one given in Lemma \ref{lem ind inner} we get 
$$(f*\linner{E}{a}{b})(\gamma)=\linner{E}{f\cdot a}{b}(\gamma).$$
So to show that $E$ is an ideal in $\A\rtimes_{\alpha,r} G$ it suffices to show that $f\cdot A_0\subset A_0.$  Now suppose $G$ acts properly on its unit space and $A_0=C_c(\units)\cdot A$ as in Proposition \ref {prpA},  if $a\in A_0$ then 
$$f\cdot a:u\mapsto \int_G f(\eta)\alpha_{\eta}(a(s(\eta)))d\haarv{}{u}(\eta).$$
But the integrand $f(\eta)\alpha_{\eta}(a(s(\eta)))$ is continuous since $f$, $a$ and the $G$ action are.  It has compact support since $f$ does.  Thus $f\cdot a$ is a continuous compactly supported section,  giving  $f\cdot a\in A_0$, so $f\cdot A_0\subset A_0$ and hence $E$ is an ideal.

A similar argument shows that in Proposition \ref{ex scalar}, the subalgebra $E\subset C_0(X)\rtimes_{\lt,r}G$ guaranteed by Theorem \ref{thm proper morita} is also an ideal of the reduced crossed product.\end{rmk}


\end{section}
\begin{section}{Saturation}
Theorem \ref{thm proper morita} guarantees that if $(\A, G, \alpha)$ is a proper dynamical system, then the generalized fixed point algebra is Morita equivalent to a subalgebra of the reduced crossed product, $\A\rtimes_{\alpha, r}G$.  This theorem is most useful when this subalgebra is itself an object we'd like to study.  In particular, we are interested in when this subalgebra is actually the algebra $\A\rtimes_{\alpha, r}G$.  So following \cite{Rieffel90} we make the following definition.
\begin{defn}We call a proper dynamical system $(\A, G, \alpha)$, {\emph{saturated}} \index{saturated} if $_{E_0}\!{A_0}_{D_0}$ completes to a $\A\rtimes_{\alpha, r}G-A^{\alpha}$ imprimitivity bimodule in Theorem \ref{thm proper morita}.\label{defn saturated}\end{defn}

Saturated dynamical systems are the proper dynamical systems primarily studied in applications \cite{aHRW00}, \cite{aHRW03}, \cite{AHRW05}, \cite{KQR08}.  So it is important to find some conditions which guarantee that a given proper action is saturated.

The goal of this section is to show the following theorem.
\begin{thm} Suppose $(\A,G, \alpha)$ is a groupoid dynamical system and let $A=\Gamma_0(\units, \A)$ be the associated $C_0(\units)$-algebra. Suppose further that $G$ is principal and proper. Then the action of $G$ on $\A$ is saturated with respect to the dense subalgebra $C_c(\units)\cdot A$. \label{thm trivsat} \end{thm}

\begin{subsection}{The Scalar Case}
In order to prove Theorem \ref{thm trivsat}, we will first show that if $G$ is principal and proper, then the action of $G$ on $C_0(\units)$ is saturated with respect to the subalgebra $C_c(\units)$.  Let $\mathscr{T}=\units\times \C$, then $\sect{\units}{\mathscr{T}}=C_0(\units)$.  Recall from Proposition \ref{ex scalar}, that the dynamical system $(C_0(\units),G,\lt)$ is proper.  To show that the action is saturated we need to show spans of elements of the form \begin{equation}\linner{E}{f}{g}(\gamma)  := f(r(\gamma))\overline{g(\gamma\inv\cdot r(\gamma))} = f(r(\gamma))\overline{g(s(\gamma))}\label{eq l inner scalar}\end{equation} are dense in $C_0(\units)\rtimes_{\lt, r}G$.  For this it suffices to show that they are dense in $\Gamma_c(G, r^*\mathscr{T})=C_c(G)$ in the inductive limit topology.  We will follow the proof in \cite{rieffel82} and construct a special approximate identity.  To construct this approximate identity, we need the following key lemma which is the groupoid analogue of \cite[Lemma p 306]{rieffel82}.   The proof follows that given in \cite{rieffel82}.

\begin{lem}Let $G$ principal and proper.  Then for each $u\in\units$ and open neighborhood $N\subset G$ of $u$,  there exists an open neighborhood $U\subset\units$ of $u$  such that $\{\gamma:\gamma\cdot U\cap U\neq \emptyset\}\subset N$.\label{lem spec nbhd}\end{lem}

\begin{proof}   By way of contradiction assume there exists an open neighborhood $N\subset G$ of $u$ such that Lemma \ref{lem spec nbhd} doesn't hold.  Then given an  open neighborhood $W\subset\units$ of $u$, there exists $\gamma_W\in G$ and $v_W\in W$ such that $\gamma_W\notin N$ and $\gamma_W\cdot v_W\in W$. For each open neighborhood $W\subset\units$ pick such a $\gamma_W\in G$ and $v_W\in W$ and order the nets $\{\gamma_W\}$ and $\{v_W\}$ by reverse inclusion.

Let $K$ be compact neighborhood of $u$ in $\units$.    Since  $(\gamma_W\cdot v_W, v_W)$ is eventually in $K\times K$, it has a convergent subnet.  By the properness of  $G$, $\{\gamma_W\}$ has a convergent subnet  $\gamma_{W_i}\rightarrow \gamma$.  Note that since $\{\gamma_{W_i}\}$ is a subnet, $W_i$ is a fundamental system for $u$, thus  $(\gamma_{W_i}\cdot v_{W_i}, v_{W_i})\rightarrow(u,u)$.  Hence $\gamma\cdot u=u$, but by assumption $\{\gamma_W\}$ is never in the open neighborhood $N$ of $u$ so $\gamma\neq u$.  This contradicts the freeness of the action.\end{proof}

Now as in \cite[p 307]{rieffel82}, we will use Lemma \ref{lem spec nbhd} to construct an approximate unit for $A$.   

\begin{lem}[Approximate Identity] Let $G$  be as in Lemma \ref{lem spec nbhd}, then there exists an approximate identity for $C_c(G)$ in the inductive limit topology given by the net $\Phi_{N,D,\epsilon}$ indexed by decreasing neighborhoods $N$ of $\units$, increasing compact subsets $D$ of $\units$, and decreasing $\epsilon>0$ which satisfies:
\begin{enumerate}
\item \label{eq app 1}$\Phi_{N,D,\epsilon}(\gamma)=0$ if $\gamma\notin N$ and $\geq 0$ otherwise,
\item \label{eq app 2}$\left|\int_G\Phi_{N,D,\epsilon}(\gamma)d\lambda^u(\gamma)-1\right|<\epsilon$ for $u\in D$.
\item \label{eq app 3}$\Phi_{N,D,\epsilon}(\gamma)=\sum{{_{{}_E}}\!\bigl\langle g_i^{N,D,\epsilon}, g_i^{N,D,\epsilon}}\bigr\rangle(\gamma)=\sum{g_i^{N,D,\epsilon}(r(\gamma))\overline{g_i^{N,D,\epsilon}(s(\gamma))}}$\\ for some $g_i^{N,D,\epsilon}\in C_c(\units).$
\end{enumerate}
\label{lem approxid}\end{lem}

\begin{proof}Let $N$ be a neighborhood of $\units$, $D$ be a compact subset of $\units$, and $\epsilon>0$ be given.  Note that  $D$ is also compact in $G$,  so we can choose an open set $V\subset G$ such that $D\subset V\subset\overline{V}\subset N$. Then using Lemma \ref{lem spec nbhd} there exists a finite open covering $\{U_i\}_{i=0}^n$ of $D$ such that for each $i$, \begin{equation}\{\gamma\in G:\gamma\cdot U_i\cap U_i\neq\emptyset\}\subset V.\label{eq def V}\end{equation}
For each $i$, pick $h_i\in C_c(U_i)$ such that $h(u):=\sum{h_i(u)}$ is strictly positive on $D$.  Now let $$m=(\inf(h|_D))/2 \quad\text{and}\quad g:=\sup(h, m).$$
Note that $m\neq 0$ since $h$ actually attains a minimum on the compact set $D$.  Furthermore, since $h$ is strictly positive on $D$, this minimum must be bigger than $0$. Thus $g>0$ on $\units$ and $g$ is continuous. Therefore \begin{equation}f_i(u):=h_i(u)/g(u)\label{eq fi}\end{equation} is in $C_c(U_i)$.

\begin{claim} If $f(u):=\sum{f_i(u)}=h(u)/g(u)$ then $0\leq f\leq 1$ and $f\equiv 1$ on~ $D$.\label{clm fdef}\end{claim}

\begin{proof}We begin by  showing $0\leq f\leq 1$.   Now if $h(u)\geq m$ then $g(u)=h(u)$ and thus $f(u)=1$.  If $h(u)\leq m$ then $g(u)=m$ and so $f(u)=h(u)/m\leq m/m=1$. Thus $f\leq 1$ and since we know each $f_i$ is positive and by definition, $m\leq h$ on $D$,  we get the result.\end{proof}

To justify the next step in the construction, suppose condition $(c)$ holds in Lemma \ref{lem approxid}. Then $$\int_G\Phi_{N, D,\epsilon}(\gamma)d\lambda^u(\gamma)=\sum{g_i^{N,D,\epsilon}(u)\int_G\overline{g_i^{N,D,\epsilon}(s(\gamma))d\lambda^u(\gamma)}}.$$  So if we could find $g_i^{N,D,\epsilon}$ such that $$f_i(\gamma)=g_i^{N,D,\epsilon}(r(\gamma))\int_G\overline{g_i^{N,D,\epsilon}(s(\eta))}d\lambda^{r(\gamma)}(\eta)$$ then we'd be done, since we know $f=\sum{f_i}\equiv 1$ on $D$.  Unfortunately we can't do that, but the next lemma shows that we can approximate $f_i$ by  functions  of the above form.  This is good enough for our purposes.

\begin{lem} Functions of the form $\gamma\mapsto g(r(\gamma))\int_G g(s(\eta))d\lambda^{r(\gamma)}(\eta)$ are dense in $C_c^+(\units)$ for the inductive limit topology.\label{lem den}\end{lem}

\begin{proof}Let $f\in C_c^+(\units)$ and $\delta>0$ be given. Define $F$ on $G\backslash\units$ by 
\begin{equation} F([u]):=\int_G f(s(\gamma)) d\lambda^u(\gamma).\label{F}\end{equation}
Let $C=\{u\in\units:~f(u)\geq\delta\}$ and let $[C]$ be the image of $C$ in $G\backslash\units$.

Let $m=\inf\{F([u]):[u]\in [C]\}$.  Since $[C]$ is compact and $F$ is continuous,  $F$ attains its minimum on $[C]$.  Furthermore, since $f$ is continuous, positive and bounded away from zero on $C$ we have that that $m>0$.  

Let $U=\{u: F([u])>m/2\}$.  By the above argument $C\subset U$.  Construct $Q\in C_c(G\backslash\units)$ such that
$0\leq Q\leq 1,$ $Q([v])=1$ for $[v]\in [C],$ and
$Q([v])=0$ for $v\notin U.$

Thus $Q/\sqrt{F}\in C_c(G\backslash\units)$.  Define $g:=fQ/\sqrt{F}$.  Then $g\in C_c^+(\units)$ and $\supp(g)\subset \supp(f)$.
Furthermore, a simple calculation shows
 $$\Bigl|f(u)-g(u)\int_G g(s(\gamma))d\lambda^u(\gamma)\Bigr|<\delta.\qedhere$$
\end{proof}
So to finish the proof of Lemma \ref{lem approxid}, let $M=\text{card}\{U_i\}$.  For each $f_i$ defined in \eqref{eq fi}, use Lemma \ref{lem den} to pick $g_i^{N,D,\epsilon}$ so that \begin{equation}\Bigl|f_i(u)-g_i^{N,D,\epsilon}(u)\int_G g_i^{N,D,\epsilon}(s(\gamma))d\lambda^u(\gamma)\Bigr|<\epsilon/M\label{eq fi gi}\end{equation} with $\supp(g_i^{N,D,\epsilon})\subset \supp(f_i).$  Now define

 \begin{equation}\Phi_{N,D,\epsilon}:=\sum{{_{{}_E}}\!\bigl\langle ~g_i^{N,D,\epsilon},~ g_i^{N,D,\epsilon}~\bigr\rangle}.\label{eq phi}\end{equation} \index{$\Phi_{N,D,\epsilon}$}
 
We need to show this $\Phi_{N,D,\epsilon}$ satisfies conditions $(a)$ and $(b)$ from Lemma \ref{lem approxid}. 

For condition \eqref{eq app 1}, notice $\supp( g_i^{N,D,\epsilon})\subset\supp( f_i)\subset U_i$, and by the definition of  $U_i$ (equation \eqref{eq def V}) $$\supp\Bigl({_{{}_E}}\!\bigl\langle ~g_i^{N,D,\epsilon},~ g_i^{N,D,\epsilon}~\bigr\rangle\Bigr)\subset\{\gamma\in G: \gamma\cdot U_i\cap U_i\neq\emptyset \}\subset V\subset N.$$  Since $i$ was arbitrary we have $\supp (\Phi_{N,D,\epsilon})\subset N$ as desired.

For property  \eqref{eq app 2}, let $u\in D$, then
\begin{align*}
\Bigl|\int_G\Bigr.&\Bigl.\Phi_{N,D,\epsilon}(\gamma)d\lambda^u(\gamma)-1\Bigr|\\
&=\Bigl|\sum{g_i^{N,D,\epsilon}(r(u))\int_G\overline{g_i^{N,D,\epsilon}(s(\gamma))}}d\lambda^u(\gamma)-1\Bigr|\\
&\leq \Bigl|\sum{-(f_i(u)-g_i^{N,D,\epsilon}(r(u))\int_G\overline{g_i^{N,D,\epsilon}(s(\gamma))}d\lambda^u(\gamma))}\Bigr|+\Bigl|\sum{f_i(u)}-1\Bigr|
\end{align*} 
which is less than $\epsilon$ by our assumptions on $g_i,~f_i$ and $D$.

It is left to show  that $\{\Phi_{N,D,\epsilon}\}$ is actually an approximate identity for $C_c(G)$ in the inductive limit topology.

First we will show that $\supp(\Phi_{N,D,\epsilon})*F$ is eventually in some compact set.  Now 
$$(\Phi_{N,D,\epsilon}*F)(\gamma)=\int_G \Phi_{N,D,\epsilon}(\eta)F(\eta\inv\gamma)d\lambda^u(\eta).$$
So for $(\Phi_{N,D,\epsilon}*F)(\gamma)\neq 0$ there is an $\eta$ such that $\eta\in \supp(\Phi_{N,D,\epsilon})\subset N$ and $\eta\inv\gamma\in \supp(F)$.  That is $\gamma\in N\cdot \supp(F)$, thus \begin{equation} \supp(\Phi_{N,D,\epsilon}*F)\subset \overline{N\cdot \supp(F)}.\label{eq suppconv}\end{equation}

To continue we need a definition, a neighborhood $W$ of $\units$ is called {\emph{diagonally compact}} (respectively {\emph{conditionally compact}})\index{conditionally compact} if $VW$ and $WV$ are compact (respectively relatively compact) for every compact (respectively compact) set $V$ in $G$.

Let $N_0$ be some open neighborhood of $\units$ then by \cite[Lemma 2.7]{MW90} there exists an open symmetric conditionally compact set $W_0$ with $\overline{W_0}$ diagonally compact, such that 
$ \units\subset W_0\subset \overline{W_0}\subset N_0.$ Thus $\overline{W_0}\cdot \supp(F)$ is compact and $\supp(\Phi_{N,D,\epsilon}*F)\subset\overline{W_0}\cdot \supp(F)$ for $N\subset W_0$. 

So it  remains to show that $\{\Phi_{N,D,\epsilon}*F\}\rightarrow F$ uniformly. Let $\delta>0$ be given.  Now compute
\begin{align}|\Phi_{N,D,\epsilon}&*F(\gamma)-F(\gamma)|=\Bigl|\int_G \Phi_{N,D,\epsilon}(\eta)F(\eta\inv\gamma)d\lambda^{r(\gamma)}(\eta)-F(\gamma)\Bigr|\nonumber\\
&\leq \int_G \left|\Phi_{N,D,\epsilon}(\eta)\right|~\left|F(\eta\inv\gamma)-F(\gamma)\right|d\lambda^{r(\gamma)}(\eta)\nonumber\\
&\hspace{1 in}+\left\|F\right\|_{\infty}\Bigl|\chi_{{}_{\supp(F)}}(\gamma)\int_G \Phi_{N,D,\epsilon}(\eta)d\lambda^{r(\gamma)}(\eta)-1\Bigr|.\label{eq tri2}
\end{align}

Notice if $r(\supp (F))\subset D$ by property $(b)$ the second term of \eqref{eq tri2} is less than $\left\|F\right\|_{\infty}\cdot\epsilon$.  So if we choose $\epsilon<\delta/(2\left\|F\right\|_{\infty})$ the second term of \eqref{eq tri2} is less than $\delta/2$.  So it remains to show that the first term is eventually less than $\delta/2$ . By way of contradiction assume $$\int_G\left|\Phi_{N,D,\epsilon}(\eta)\right|~\left|F(\eta\inv\gamma)-F(\gamma)\right|d\lambda^{r(\gamma)}(\eta)\geq \delta/2 \quad\forall~ (N, D,\epsilon).  $$  Note that if we choose $W_0$ as above, and if $N\subset W_0$ then for $\gamma\notin (\overline{W_0}\cdot \supp(F)\cup \supp(F))$  the first term of \eqref{eq tri2} is $0$.  Thus we can restrict our attention to when $\gamma\in (\overline{W_0}\cdot \supp(F)\cup \supp(F))$ and notice this set is compact since it is the union of two compact sets. 

\begin{claim} There exists an open neighborhood $N$ of $\units$ such that for $\gamma\in (\overline{W_0}\cdot \supp(F)\cup \supp(F))$, $\eta\in N$  we have $\left|F(\eta\inv\gamma)-F(\gamma)\right|<\delta/4$.  \label{clm Fbd}\end{claim}

\begin{proof}By way of contradiction assume the claim is false.  Thus for each neighborhood $N$ of $\units$ we can choose $\gamma_{{}_N}\in  (\overline{W_0}\cdot \supp(F)\cup \supp(F))$ and $\eta_{{}_N}\in N$ such that $\left|F(\eta_{{}_N}\inv\gamma_{{}_N})-F(\gamma_{{}_N})\right|\geq\delta/4$.  Since $\gamma_{{}_N}$ is a net in a compact set it has a convergent subnet which by relabeling we can assume $\gamma_{{}_N}\rightarrow\gamma$.  Also take the corresponding subnet of $\eta_{{}_N}$.

Pick an $r$-relatively compact neighborhood $N_0\subset W_0$ of $\units$, and set $K=r(\overline{W_0}\cdot \supp(F)\cup \supp(F))$. Then $r\inv(K)\cap N_0$ is relatively compact and for $N\subset N_0$,  $r(\eta_{{}_N})=r(\gamma_{{}_N})\in K$.  Thus $\eta_{{}_N}\in r\inv(K)\cap N_0$ which is relatively compact by assumption.  Thus $\eta_{{}_N}$ must have a convergent subnet  $\eta_{{}_{N_i}}\rightarrow \eta$.  By our choice of $\eta_{{}_{N_i}}$ we must have $\eta\in \units$. Choose this subnet of $\eta_{{}_N}$ and the corresponding subnet of $\gamma_{{}_N}$ and relabel.  Thus $\eta_{{}_N}\inv\gamma_{{}_N}\rightarrow\eta\inv\gamma=\gamma$, hence $\left|F(\eta_N\inv\gamma_N)-F(\gamma_N)\right|\rightarrow \left|F(\gamma)-F(\gamma)\right|= 0$ a contradiction.  \end{proof}

Now a simple computation using \eqref{eq fi gi} shows that

\begin{claim} For $\epsilon<1$,  the integral $\int_G \left|\Phi_{N,D,\epsilon}(\eta)\right|d\lambda^u(\eta)<2$.\label{clm intbd}\end{claim}

 Thus if we pick $N_0$ as in Claim \ref{clm Fbd}, $D=\supp (F)$ and $\epsilon=\delta/(2\left\|F\right\|_{\infty})$, then by the discussion after \eqref{eq tri2},
\begin{align*}\left|\Phi_{N,D,\epsilon}*F(\gamma)-F(\gamma)\right|
&<\int_G \left|\Phi_{N,D,\epsilon}(\eta)\right|~\left|F(\eta\inv\gamma)-F(\gamma)\right|d\lambda^{r(\gamma)}(\eta)+\delta/2\\
&<\int_G \left|\Phi_{N,D,\epsilon}(\eta)\right|(\delta/4 )d\lambda^{r(\gamma)}(\eta)+\delta/2\leq \delta
\end{align*} by Claim \ref{clm Fbd} and property (a).
Hence $\Phi_{N,D,\epsilon}$ is an approximate identity for $A$ in the inductive limit topology so Lemma \ref{lem approxid} is proved.\end{proof}
 
Lemma \ref{lem approxid} shows that $\spn\{\linner{E}{f}{g}\}$ is  dense in $C_c(G)$ in the inductive limit topology and thus is  dense in $C^*_r(G)=C_0(\units)\rtimes_{\lt, r} G$ giving us the following theorem.

\begin{thm}\label{thm sat scal}Suppose $G$ is principal and proper. Then the dynamical system $(C_0(\units),G,\lt)$ is saturated with respect to the dense subalgebra $C_c(\units)$, that is $C_0(G\backslash\units)\cong C_0(\units)^{\lt}$ is Morita equivalent to $C_r^*(G)$.\end{thm}

Note that this Theorem implies that the spectrum of  $C_r^*(G)$ is $G\backslash\units$ and furthermore that $C_r^*(G)$ is \emph{globally} Morita equivalent to $C_0(G\backslash\units)$. So  by applying \cite[Proposition 5.15]{tfb} we get the following result.

\begin{cor}If $G$ is a second countable groupoid acting freely and properly on its unit space, then $C_r^*(G)$ has continuous trace with trivial Dixmier-Douady invariant.\label{cor cts trace scal}\end{cor}

\begin{rmk}Since $G$ is principal and proper, by \cite[Corollary 2.1.17]{AR00}, $G$ is properly amenable.  Thus \cite[Definition 2.1.13]{AR00} and \cite[Definition 2.2.2]{AR00} show that it is  topologically amenable.  So by \cite[Proposition 3.35]{AR00} we have that $G$ is measure wise amenable.  Thus \cite[Proposition 6.1.8]{AR00} gives that $C^*(G)=C^*_{\text{red}}(G)$.  We use the notation $C^*_{\text{red}}(G)$ here because it is \emph{a priori} different from $C^*_r(G)$ defined in Section  \ref{sec red cross} and \cite[Definition II.2.8]{Ren80}.  In \cite{AR00}, $\|\cdot\|_{\text{red}}$  is  defined using only those representations induced by point mass measures on $\units$, therefore $$\|\cdot\|_{\text{red}}\leq \|\cdot\|_r\leq \|\cdot\|_{\text{universal}}.$$  
Now \cite[Proposition 6.1.8]{AR00} implies that $\|\cdot\|_{\text{red}}$ is the same as the universal norm, thus $\|\cdot\|_r$ must be the same as the universal norm as well and hence $C_r^*(G)=C^*(G)$.  Thus Corollary \ref{cor cts trace scal} recovers \cite[Proposition 2.2]{MW90}.\label{rmk amen grp}\end{rmk}
\begin{rmk}If $G=H\times X$ is a transformation group groupoid the condition $G$ acts freely and properly on its unit space means that $H$ acts freely and properly on $X$.  Therefore, Corollary \ref{cor cts trace scal} and Remark \ref{rmk amen grp} give \cite[Corollary 15]{Grn77}.\end{rmk}
\end{subsection}
\begin{subsection}{Proof of Theorem \ref{thm trivsat}}

We now prove Theorem \ref{thm trivsat}, which states that if a groupoid $G$ acts freely and properly on its unit space then the action of $G$ on any upper-semicontinuous $C^*$-bundle is saturated, that is finite linear combinations of  the inner product $$\linner{E}{ f\cdot a}{ g\cdot b}:=\gamma\mapsto f(r(\gamma))a(r(\gamma))\alpha_{\gamma}((g(s(\gamma))b(s\gamma))^*)$$ $f,g\in C_c(\units),~a,b\in A=\Gamma_0(\units,\A)$ are dense in $\A\rtimes_{\alpha,r}G$. This proof  follows \cite[Appendix C]{AHRW05} fairly closely.
We proceed in several steps.  The first two steps show that we can consider functions of compact support.  Then we cover the support of the function we want to approximate by small enough neighborhoods, so that the action on these neighborhoods is almost trivial, and finally we use a partition of unity to complete the approximation.  
\\~\\
\underline{Step I}:~~  Show that the span of sections of the form 

\begin{equation}F(\gamma)=\phi(\gamma)f(r(\gamma))a(r(\gamma))g(r(\gamma))b(r(\gamma))^*\label{woalpha}\end{equation} 
are dense in $\Gamma_c(G, r^*\A)$ in the inductive limit topology, where $\phi\in C_c(G),~f,g\in C_c(\units),\text{~and}~a,b\in A=\Gamma_0(\units,\A)$.

To see this first note that  $A_0^2$ is dense in $A$.  Now from \cite[Proposition 1.3]{RW85}, we know that $$C_0(G)\otimes_{C_0(\units)}\Gamma_0(\units, \A)\cong \Gamma_0(G,r^*\A)$$ where the isomorphism given on elementary tensors by $$\Phi : f\otimes a\mapsto \bigl(\gamma\mapsto f(\gamma)a(r(\gamma))\bigr).$$

Note that $\Phi(C_c(G)\odot A_0^2)$ is a $C_0(G)$-module.  Furthermore, since $A_0^2$ is dense in $A$ we have $A_0^2(r(\gamma))$ is dense in $A(r(\gamma))$. Thus by \cite[Proposition C.24]{TFB2}, $\Phi(C_c(G)\odot A_0^2)$ is dense in $\Gamma_0(G,r^*\A)$ and hence in $\Gamma_c(G,r^*\A)$ in the uniform topology.  
 
Given $\psi\in \Gamma_c(G,r^*\A)$ pick a net $\psi_j'\rightarrow \psi$ uniformly with $\psi_j'\in \Phi(C_c(G)\odot A_0^2)$.  Pick $\omega\in C_c(G)$ such that $0\leq\omega\leq 1$ and $\omega\equiv 1$ on $\supp(\psi)$.  Then $\psi=\omega\psi=\lim\omega\psi_j'$.  Let $\psi_j=\omega\psi_j'$ then $\psi_j\rightarrow\psi$ uniformly and $\supp (\psi_j)\subset \supp (\omega)$ which is compact.  Thus $\psi_j\rightarrow\psi$ in the inductive limit topology. 

Note that every element of $\Phi(C_c(G)\odot A_0^2)$ is of the form \eqref{woalpha}. Thus it suffices to show that elements of the form \eqref{woalpha} can be approximated by elements of $E$ in the inductive limit topology.
~\\~\\
\underline{Step II}:~~  Show that elements of the form 
\begin{equation}\gamma\mapsto\phi(\gamma)f(r(\gamma))a(r(\gamma))\alpha_{\gamma}\left(\overline{g(s(\gamma))}b^*(s(\gamma))\right)\label{walpha}\end{equation}
are in $E$ with $\phi\in C_c(G),~f,g\in C_c(\units), ~a,b\in A$.

Now by Theorem \ref{thm sat scal}, the action of $G$ on $C_0(\units)$ is saturated with respect to $C_c(\units)$.  Thus, given $\epsilon >0$, we can find $g_i, h_i\in C_c(\units)$ such that

\begin{equation}\Bigl\|~\phi(\gamma)-\sum_i{g_i(r(\gamma))\overline{h_i(s(\gamma))}~}~\Bigl\|<\frac{\epsilon}{\|f\|_{\infty}\|a\|\|b\|\|g\|_{\infty}}.\label{blah}\end{equation}

Furthermore, we can arrange it so that if $W$ is a compact neighborhood of the support of $\phi$, then $\supp\left(\gamma\mapsto \sum_i{g_i(r(\gamma))\overline{h_i(s(\gamma))}}\right)\subset W$.  Now let $a_i=g_i f\cdot a ~\text{and}~b_i=h_ig\cdot b$ then
\begin{align*}
\Bigl\|~\sum_i&\linner{E}{a_i}{b_i}(\gamma)-\phi(\gamma)f(r(\gamma))a(r(\gamma))\alpha_{\gamma}\left(\overline{g(s(\gamma))}b^*(s(\gamma))\right)~\Bigr\|\\
&=\Bigl\|\sum_i{g_i(r(\gamma))\overline{h_i(s(\gamma))}}\linner{E}{f\cdot a}{g\cdot b}(\gamma)-\phi(\gamma)\linner{E}{f\cdot a}{g\cdot b}\Bigr\|\\
&\leq\Bigl\|\sum_i{g_i(r(\gamma))\overline{h_i(s(\gamma))}}-\phi(\gamma)\Bigr\|\left\|\linner{E}{ f\cdot a}{g\cdot b}\right\| <\frac{\|f\|_{\infty}\|a\|\|b\|\|g\|_{\infty}\epsilon}{\|f\|_{\infty}\|a\|\|b\|\|g\|_{\infty}}=\epsilon.\end{align*}
Since $W$ does not depend on $\epsilon$ and $\epsilon$ is arbitrary, we must have $$\gamma\mapsto\phi(\gamma)f(r(\gamma))a(r(\gamma))\alpha_{\gamma}(\overline{g(s(\gamma))}b^*(s(\gamma)))\in E.$$
~\\
\underline{Step III}:~~ Show that the functions of the form \eqref{walpha} can be used to approximate the functions of the form \eqref{woalpha} in the inductive limit topology.

\begin{rmk}At this point in \cite[Appendix C]{AHRW05}, the authors find a neighborhood $N$ of the identity in the group such that $\|b^*-\alpha_s(b^*)\|$ is small for $s\in N$. They then translate this neighborhood to find a finite collection of open sets $Nr_i$ such that $\supp(\phi)\subset \bigcup Nr_i$.  They use this open cover to construct a partition of unity, $\phi_i$ and define
\begin{equation}F_i(s):=\phi(s)\phi_i(s)(f\cdot a)\alpha_{sr_i\inv}((g\cdot b)^*).\label{eq ap aH etal}\end{equation}

This is a fairly standard approximation argument in group crossed products. Unfortunately, this argument does not work for groupoids, since the translation of an open set in a groupoid by a groupoid element is not necessarily open.  However, the translation $UV_i$ of an open set $U\subset G$ by an open set $V_i\subset G$ is open in a groupoid.  But now we do not have an element $r_i$ to plug into an analogous equation to \eqref{eq ap aH etal}.  The idea which motivates what follows is to average $\alpha_{\gamma\eta\inv}((g(r(\eta))\cdot b(r(\eta))^*)$ over $\eta\in V_i$.\end{rmk}

Fix $$F(\gamma)=\phi(\gamma)f(r(\gamma))a(r(\gamma))g(r(\gamma))b(r(\gamma))^*$$ as in \eqref{woalpha} with $\phi\in C_c(G),~f,g\in C_c(\units),~a,b\in A=\Gamma_0(\units,\A)$ and let $\epsilon>0$ be given.  Define
\begin{equation}K:=\supp(\phi).\label{supp}\end{equation}

Note that since the norm is upper semicontinuous, the set
\begin{equation}N_{\epsilon}:=\Bigl\{\gamma:\Bigl\|b^*(r(\gamma))\overline{g(r(\gamma))}-\alpha_{\gamma}\bigl(b^*(s(\gamma))\overline{g(s(\gamma))}\bigr)\Bigr\|<\frac{\epsilon}{\|\phi\|\|f\|\|a\|}\Bigr\}\label{Nepsilon}\end{equation}
is open.  Furthermore, it is nonempty since $\units\subset N_{\epsilon}$.  Now we need a lemma whose proof follows easily from the continuity of multiplication in $G$.

\begin{lem}For every $\eta\in G$, there exists an open neighborhood $U_{\eta}$ of $\eta$ such that $U_{\eta}\cdot U_{\eta}\inv\subset N_{\epsilon}$.\label{ueta}\end{lem}

Now for $\eta\in K=\supp(\phi)$, let $U_{\eta}$ be an open neighborhood of $\eta$ as in Lemma \ref{ueta}.  Then $\{U_{\eta}\}_{\eta\in K}$ is an open cover of the compact set $K$, therefore there is a finite subcover $\{U_{i}\}_{i=1}^n$.  Furthermore, since $G$ is locally compact and Hausdorff, $G$ is regular.  Thus, for all $\eta\in K$ there exists a neighborhood $V_{\eta}$ of $\eta$ with compact closure such that $\eta\in V_{\eta}\subset\overline{V_{\eta}}\subset U_i$ for some $i$.  Now $\{V_{\eta}\}_{\eta\in K}$ is an open cover of the compact set $K$, therefore there is a finite subcover $\{V_{j}\}_{j=1}^m$.  We have arranged it so that for each $j$ there exists $i$  such that $\overline{V_j}\subset U_i$.  For each $j=1,\cdots, m$ pick such an $i$ and define $$\sigma:\{1,\ldots ,m\}\rightarrow \{1,\ldots, n\} \text{ so that }  \overline{V_j}\subset U_{\sigma(j)}.$$

For each $V_j$ pick a function $\psi_j$ such that $0\leq \psi_j\leq 1$, $\psi_j|_{\overline{V_j}}\equiv 1$ and $\supp(\psi_j)\subset U_{\sigma(j)}$.  
Also, pick a partition of unity $\{\phi_j\}\subset C_c^{+}(G)$ subordinate to the subcover $\{V_{j}\}_{j=1}^m$.  That is, $\supp(\phi_i)\subset V_j$,~ $0\leq \phi_j\leq 1$ ,~ $\sum{\phi_j}\equiv 1$ on $K$ and $\sum{\phi_j}\equiv 0$ off of $\bigcup V_j$.  We will use these functions to ensure that groupoid elements lie in $N_{\epsilon}$.

Define \begin{equation} \omega_j(u):=\int_G \psi_j(\gamma)d\lambda_u(\gamma).\label{omegaj}\end{equation}
Now $\omega_j$ is continuous since $\psi_j\in C_c(G)$ and $\lambda_u$ is a (right) Haar system.  Furthermore, $$\gamma\mapsto \frac{\phi_j(\gamma)}{\omega_j(s(\gamma))}$$ is continuous, since $$\supp(\phi_j)\subset V_j\subset \overline{V_j}\subset \supp(\psi_j)\subset \supp(w_j\circ s).$$

Now we are ready to define the functions we will use to approximate $F$.  
\begin{align}f_j(\gamma)&:= \phi(\gamma)\frac{\phi_j(\gamma)}{\omega_j(s(\gamma))} f(r(\gamma))a(r(\gamma))\nonumber\\
&\hspace{.7 in} \alpha_{\gamma}\Bigl(\int_G\psi_j(\eta)~\alpha_{\eta\inv}\left(b^*(r(\eta))\overline{g(r(\eta))}\right)d\lambda_{s(\gamma)}(\eta)\Bigr)
\label{fj}\end{align}

\begin{rmk}The functions $f_j$ are of the form  of equation \eqref{walpha}.

To see this notice $\gamma\mapsto \phi(\gamma)\frac{\phi_j(\gamma)}{\omega_j(s(\gamma))}\in C_c(G)$ and that $\alpha_{\eta\inv}\bigl(b^*(r(\eta))\overline{g(r(\eta))}\bigr)\in A(s(\eta))$.  Thus  $$\int_G\psi_j(\eta)~\alpha_{\eta\inv}\left(b^*(r(\eta))\overline{g(r(\eta))}\right)d\lambda_u(\eta)\in A(u).$$
 But $$\psi_j(\eta)~\alpha_{\eta\inv}\left(b^*(r(\eta))\overline{g(r(\eta))}\right)\in C_c(G)$$ since $\psi_j$ is.  So since $\lambda$ is  a Haar system we have that $$\Bigl(u\mapsto\int_G\psi_j(\eta)~\alpha_{\eta\inv}\left(b^*(r(\eta))\overline{g(r(\eta))}\right)d\lambda_u(\eta)\Bigr)\in \Gamma_c(\units, \A)$$ and thus is of the form $\overline{h(s(\gamma))}c^*(s(\gamma))$ for $c\in A$ and $h\in C_c(\units)$.
 \label{rmk fj form}\end{rmk}

\begin{rmk}If $\gamma\in \supp(f_j)$ then  $\gamma\in \supp(\phi_j)$, so that $\gamma\in U_{\sigma(j)}$. Now if the integrand in \eqref{fj} is nonzero,  then $\eta\in \supp(\psi_j)$ so that $\eta\in U_{\sigma(j)}$. That is $\gamma\eta\inv\in N_{\epsilon}$.
\label{rmk eta gamma in N}\end{rmk}

\begin{rmk} Now $\supp(f_j)\subset\supp(\phi)=K$, so that $\supp(\sum f_j)\subset K$.\label{supp sum fj}\end{rmk}
 
To finish the proof we compute:
$$\left\|F(\gamma)-\sum{(f_j(\gamma))}\right\|=\left\|\sum{\left(\phi_j(\gamma)F(\gamma)-f_j(\gamma)\right)}\right\|$$
since $\sum{\phi_j(\gamma)}\equiv 1$ on $K\supset \supp(F)\cap \supp(\sum{f_j})$. 

\begin{align*}
&=\left\| \sum\phi(\gamma)\phi_j(\gamma)f(r(\gamma))a(r(\gamma))\left(b^*(r(\gamma))\overline{g(r(\gamma))}\right.\right. \\ 
&\hspace{1 in}\Bigl.\Bigl. -\int_G\frac{\psi_j(\eta)}{\omega_j(s(\gamma))}~\alpha_{\gamma\eta\inv}\left(b^*(r(\eta))\overline{g(r(\eta))}\right)d\lambda_{s(\gamma)}(\eta)\Bigr) \Bigr\|\\
&\leq \|\phi\|_{\infty}\|f\|_{\infty}\|a\|\chi_{{}_K}(\gamma)\left\|\sum\phi_j(\gamma)\left(b^*(r(\gamma))\overline{g(r(\gamma))}\right.\right.\\
&\hspace{1 in} \Bigl.\Bigl. -\int_G\frac{\psi_j(\eta)}{\omega_j(s(\gamma))}~\alpha_{\gamma\eta\inv}\left(b^*(r(\eta))\overline{g(r(\eta))}\right)d\lambda_{s(\gamma)}(\eta)\Bigr)\Bigr\|\\
&\leq \|\phi\|_{\infty}\|f\|_{\infty}\|a\|\chi_{{}_K}(\gamma)\sum\Bigl(\phi_j(\gamma) \Bigl\|\Bigl(\int_G\frac{\psi_j(\eta)}{\omega_j(s(\gamma))}~b^*(r(\gamma))\cdot\Bigr.\Bigr.\Bigr.\\
&\hspace{.2 in}\Bigl.\Bigl.\Bigl. \overline{g(r(\gamma))}d\lambda_{s(\gamma)}(\eta)-\int_G\frac{\psi_j(\eta)}{\omega_j(s(\gamma))}~\alpha_{\gamma\eta\inv}\left(b^*(r(\eta))\overline{g(r(\eta))}\right)d\lambda_{s(\gamma)}(\eta)\Bigr)\Bigr\|\Bigr)
\end{align*}
since $\int_G\frac{\psi_j(\eta)}{\omega_j(u)}d\lambda_{u}(\eta)\equiv 1$ on $\supp(\phi_j)$ and $b^*(r(\gamma))\overline{g(r(\gamma))}$ doesn't depend on $\eta$.

\begin{align}&\leq \|\phi\|_{\infty}\|f\|_{\infty}\|a\|\chi_{{}_K}(\gamma)\sum \Bigl(\phi_j(\gamma)\int_G\frac{\psi_j(\eta)}{\omega_j(s(\gamma))}\Bigr.\nonumber\\
&\left.\hspace{.3 in} \left\|~b^*(r(\gamma))\overline{g(r(\gamma))}-~\alpha_{\gamma\eta\inv}\left(b^*(r(\eta))\overline{g(r(\eta))}\right)\right\|d\lambda_{s(\gamma)}(\eta)\right).\label{normint}\end{align}
But by Remark \ref{rmk eta gamma in N} and equation \eqref{Nepsilon}, we know that 
$$\left\|~b^*(r(\gamma))\right.\left.\overline{g(r(\gamma))}-\alpha_{\gamma\eta\inv}\left(b^*(r(\eta))\overline{g(r(\eta))}\right)\right\|<\frac{\epsilon}{\|\phi\|_{\infty}\|f\|_{\infty}\|a\|}.$$
So that \eqref{normint} is less than
\begin{align*}
&\|\phi\|_{\infty}\|f\|_{\infty}\|a\|\chi_{{}_K}(\gamma)\sum \Bigl(\frac{\epsilon}{\|\phi\|_{\infty}\|f\|_{\infty}\|a\|}\phi_j(\gamma)\int_G\frac{\psi_j(\eta)}{\omega_j(s(\gamma))} d\lambda_{s(\gamma)}(\eta)\Bigr)\\
&=\|\phi\|_{\infty}\|f\|_{\infty}\|a\|\frac{\epsilon}{\|\phi\|_{\infty}\|f\|_{\infty}\|a\|}=\epsilon
\end{align*}
since $\int_G\frac{\psi_j(\eta)}{\omega_j(u)}d\lambda_{u}(\eta)\equiv 1$ on $\supp(\phi_j)$ and $\sum{\phi_j(\gamma)}\equiv 1$ on $K$.

Thus we  can approximate $F$ by $\sum f_j$ in the inductive limit topology.  Now Steps I and II along with Remark \ref{rmk fj form} gives the density of  $\spn\{\linner{E}{f\cdot a}{g\cdot b}:f,g\in C_c(\units), a,b\in A\}$  in $\Gamma_c(G,r^*\A)$ in the inductive limit topology and hence $\spn\{\linner{E}{f\cdot a}{g\cdot b}:f,g\in C_c(\units), a,b\in A\}$ is dense in $\A\rtimes_{\alpha, r}G$.  Thus the dynamical system $(\A,G,\alpha)$ is saturated and we obtain Theorem \ref{thm trivsat}.

\begin{rmk} As in Remark \ref{rmk amen grp} for $G$ acting freely and properly on its unit space, $\A\rtimes_{\alpha, r}G=\A\rtimes_{\alpha} G$ using \cite[Proposition 6.1.10]{AR00} or \cite[Theorem 3.6]{Ren91}. So that $A^{\alpha}$ is Morita equivalent to $\A\rtimes_{\alpha} G$.\end{rmk}
\end{subsection}
\end{section}
\providecommand{\bysame}{\leavevmode\hbox to3em{\hrulefill}\thinspace}
\providecommand{\MR}{\relax\ifhmode\unskip\space\fi MR }
\providecommand{\MRhref}[2]{%
  \href{http://www.ams.org/mathscinet-getitem?mr=#1}{#2}
}
\providecommand{\href}[2]{#2}

\end{document}